\documentclass[a4paper,11pt]{article}

\title{Automorphism groups of linearly ordered structures and
  endomorphisms of the ordered set~$\Qset$ of rational numbers}
\author{Jillian~D. McPhee, James~D. Mitchell \& Martyn Quick\\[8pt]
  Mathematical Institute, University of St Andrews, North Haugh,\\
  St Andrews, Fife, KY16 9SS, United Kingdom}

\usepackage{amsmath,amssymb}
\usepackage{mathrsfs}
\usepackage{theorem}
\usepackage[margin=3cm]{geometry}

\newcommand{\1}{\mathbf{1}}
\newcommand{\C}{\mathscr{C}}
\newcommand{\card}[1]{\mathopen{|}#1\mathclose{|}}
\newcommand{\nbd}{\nobreakdash-}
\newcommand{\N}{\mathbb{N}}
\newcommand{\oset}[1]{(#1,{\leq})}
\newcommand{\Q}{\mathbb{Q}}
\newcommand{\Qset}{\oset{\Q}}
\newcommand{\R}{\mathbb{R}}
\newcommand{\scD}{\mathscr{D}}
\newcommand{\scH}{\mathscr{H}}
\newcommand{\scL}{\mathscr{L}}
\newcommand{\scR}{\mathscr{R}}
\newcommand{\set}[2]{\left\{\,#1\;\middle|\;#2\,\right\}}
\newcommand{\T}{\mathcal{T}}
\newcommand{\x}{\mathbf{x}}
\newcommand{\y}{\mathbf{y}}
\newcommand{\Z}{\mathbb{Z}}
\DeclareMathOperator{\Aut}{Aut}
\DeclareMathOperator{\End}{End}
\DeclareMathOperator{\im}{im}
\DeclareMathOperator{\Sym}{Sym}

\renewcommand{\emptyset}{\varnothing}
\renewcommand{\epsilon}{\varepsilon}
\renewcommand{\geq}{\geqslant}
\renewcommand{\leq}{\leqslant}

\newcommand{\spc}{\vspace{\baselineskip}}

\renewcommand{\labelenumi}{{\normalfont(\roman{enumi})}}

\theorembodyfont{\slshape}
\newtheorem{THM}{Theorem}

\newtheorem{thm}{Theorem}[section]
\newtheorem{lemma}[thm]{Lemma}
\newtheorem{prop}[thm]{Proposition}
\newtheorem{cor}[thm]{Corollary}

\newcommand{\qed}{\hspace*{\fill}$\square$}
\newenvironment{proof}{%
  \begin{trivlist}\item\textsc{Proof:}}{\qed\end{trivlist}}
\newenvironment{proofof}[1]{%
  \begin{trivlist}\item\textsc{Proof of #1:}}{%
  \qed\end{trivlist}}

\begin{document}

\maketitle

\begin{abstract}
  We investigate the structure of the monoid of endomorphisms of the
  ordered set~$\Qset$ of rational numbers.  We show that for any
  countable linearly ordered set~$\Omega$, there are uncountably many
  maximal subgroups of~$\End\Qset$ isomorphic to the automorphism
  group of~$\Omega$.  We characterise those subsets~$X$ of~$\Q$ that
  arise as a retract in~$\Qset$ in terms of topological information
  concerning~$X$.  Finally, we establish that a countable group arises
  as the automorphism group of a countable linearly ordered set, and
  hence as a maximal subgroup of~$\End\Qset$, if and only if it is
  free abelian of finite rank.
\end{abstract}

\section{Introduction}

The linearly ordered set~$\Qset$ of rational numbers has been observed
to have a number of interesting properties.  From the model theory
point of view, $\Qset$~is the Fra\"{\i}ss\'{e} limit of the class of
finite linearly ordered sets.  In addition, the automorphism group
of~$\Qset$ is highly homogeneous as a permutation group on~$\Q$ and it
is oligomorphic (see, for example,~\cite[Section~9.3]{BMMN}).  In this
paper, we present an investigation into the structure of the
endomorphism monoid of~$\Qset$.  This continues our work begun
in~\cite{Jay1}, where we examined the endomorphism monoid of the
random graph and other related relational structures that arise as
Fra\"{\i}ss\'{e} limits of various types of finite graph.  The content
of this paper then presents the beginnings of a counterpart within the
study of the monoid of endomorphisms of~$\Qset$ to the literature on
its automorphism group.

Our main result (Theorem~\ref{thm:main} below) is to describe the
groups that arise as maximal subgroups of the endomorphism monoid
of~$\Qset$.  The identity element in each such maximal subgroup~$G$ is
an idempotent endomorphism~$f$ of~$\Qset$ and the group~$G$ is then
equal to the group $\scH$\nbd class of~$f$.  We identify the
group~$G$, up to isomorphism, from the structure of the image of~$f$
(see Lemma~\ref{lem:H-class} below).  It is therefore important to
identify which linearly ordered sets arise as \emph{retracts} (that
is, the image of an idempotent endomorphism) in~$\Qset$ and, within
the first author's thesis~\cite{JayThesis}, the original proof of
Theorem~\ref{thm:main} relied upon a precise description of such
retracts as subsets of~$\Q$.  Recently, Kubi\'{s}~\cite{Kubis} has
studied retracts of Fra\"{\i}ss\'{e} limits in wider generality and he
establishes the following theorem.

\begin{thm}[Kubi\'{s}, {\protect\cite[Corollary~3.24]{Kubis}}]
  \label{thm:Kubis}
  Every countable linearly ordered set is order-isomorphic to an
  (increasing) retract of the set of rational numbers.
\end{thm}
(``Increasing'' in this statement merely refers to the fact that
Kubi\'{s} uses the term ``increasing map'' for what we call an
order-preserving map.)  We shall use this result of Kubi\'{s}
throughout Section~\ref{sec:EndQ} and it will enable us to present
considerably shorter proofs than otherwise possible, essentially
because we can always start with some endomorphism of~$\Qset$ with
image order-isomorphic to some particular linearly ordered set.

Our main result is the following where we describe the groups arising
as the maximal subgroups of~$\End\Qset$.  There are some obvious
restrictions on such groups: they must act as automorphisms on the
image of the corresponding retract; i.e., they must be an automorphism
group of some linearly ordered set.  In the theorem, we observe that
this is the only restriction.  Moreover, we also show that each group
occurs uncountably many times as a maximal subgroup of~$\End\Qset$.
(Note that the trivial group occurs as the group $\scH$\nbd class of
many idempotent endomorphisms including all those with finite image.)

Green's relations are used to describe the structure of a semigroup
and the theorem is expressed in terms of these relations.  We describe
them in more detail in Section~\ref{sec:prelims} below.  The
$\scD$\nbd relation is the coarsest of those that we consider, while
$\scD$\nbd classes are refined into $\scL$- and $\scR$\nbd classes.
Finally $\scH = \scL \cap \scR$.  Maximal subgroups of a monoid are
found as the \emph{group $\scH$\nbd classes} within \emph{regular}
$\scD$\nbd classes and the idempotent elements within these classes
play the role of the identity element in each maximal subgroup.

\begin{THM}
  \label{thm:main}
  \begin{enumerate}
  \item Let $\Omega$~be a countable linearly ordered set.  Then there
    exist $2^{\aleph_{0}}$~distinct regular $\scD$\nbd classes
    of\/~$\End\Qset$ whose group $\scH$\nbd classes are isomorphic
    to~$\Aut\Omega$.
  \item There is one countable regular $\scD$\nbd class~$D_{0}$
    of\/~$\End\Qset$.  This~$D_{0}$ consists of the (idempotent)
    endomorphisms with image of cardinality~$1$ and every $\scH$\nbd
    class in~$D_{0}$ is a group $\scH$\nbd class isomorphic to the
    trivial group.

    All other regular $\scD$\nbd class of\/~$\End\Qset$ contain
    $2^{\aleph_{0}}$~distinct group $\scH$\nbd classes.
  \end{enumerate}
\end{THM}

The strategy for proving Theorem~\ref{thm:main} is quite similar to
the corresponding theorems for Fra\"{\i}ss\'{e} limits of graphs
in~\cite{Jay1}.  We shall construct uncountably many linearly ordered
sets~$\C_{\x}$ with trivial automorphism group that we ``attach'' to
the linearly ordered set~$\Omega$ in such a way that the result has
automorphism group isomorphic to that of~$\Omega$.  There are two
differences to note.  The first is that the construction of the
ordered set~$\C_{\x}$ is relatively delicate, whereas the
graph~$L_{\Sigma}$ with trivial automorphism group in~\cite{Jay1} is
quite easy to build.  On the other hand, having constructed~$\C_{\x}$,
the use of Kubi\'{s}'s result makes it now straightforward to find an
idempotent endomorphism with specified image.

We shall also establish further structural information concerning the
endomorphism monoid of~$\Qset$ and its elements, as follows:
\begin{itemize}
\item We determine the number of $\scR$- and $\scL$\nbd classes in
  each $\scD$\nbd class and observe that this depends upon the
  cardinality of the image of any endomorphism within the $\scD$\nbd
  class (Theorems~\ref{thm:R-classes} and~\ref{thm:L-classes}).
\item If an endomorphism has finite image then it is a regular element
  of~$\End\Qset$ \ (Proposition~\ref{prop:finite->regular}), while
  non-regular endomorphisms can be constructed with certain types of
  infinite image (Theorem~\ref{thm:nonregular}) from which it follows
  there are uncountably many non-regular $\scD$\nbd classes
  (Corollary~\ref{cor:uc-nonreg}).
\end{itemize}
Note that the first of these sets of observations applies to all
$\scD$\nbd classes in~$\End\Qset$, which stands in contrast to our
work in~\cite{Jay1}, where we established similar results for all
\emph{regular} $\scD$\nbd classes in the endomorphism monoid of the
random graph and constructed \emph{some} examples of non-regular
$\scD$\nbd classes with uncountably many $\scR$- and $\scL$\nbd
classes.

\spc

Kubi\'{s}'s result states that every countable linear ordered set is
isomorphic to the image of some idempotent endomorphism of~$\Qset$.
There is still the question of which subsets of~$\Q$ actually arise as
the image of an idempotent endomorphism.  This is the content of the
second theorem we state in this introduction, where we characterise
precisely which subsets are the images of idempotent endomorphisms.
This theorem was used in the original proof of Theorem~\ref{thm:main}
as it appeared in~\cite{JayThesis}.  Although no longer needed for the
proofs in Section~\ref{sec:EndQ}, this characterisation still seems to
be of interest, particularly in the context of Theorem~\ref{thm:Kubis}
quoted above from~\cite{Kubis}.

To state the theorem, we need the following concept.  If $A$~is a
subset of~$\Q$, define a symmetric relation~$\sim$ on~$A$ as follows.
If $a,b \in A$ with $a \leq b$, define $a \sim b$ if $c \in A$ for
all~$c$ satisfying $a \leq c \leq b$.  It is straightforward to verify
that $\sim$~is an equivalence relation on~$A$.  We shall call the
equivalence classes of~$A$ under this relation the \emph{maximal
  intervals} of~$A$.  Indeed, it follows from the definition that
every equivalence class of~$\sim$ is an interval (as defined in
Section~\ref{sec:prelims} below) contained within~$A$ and if $J$~is
any interval contained in~$A$ then $a \sim b$ for all~$a,b \in J$ and
hence $J$~is contained within one of the equivalence classes.

\begin{THM}
  \label{thm:image}
  Let $X$~be a subset of~$\Q$.  Then there is an idempotent
  endomorphism of~$\Qset$ with image equal to~$X$ if and only if no
  maximal interval within~$\Q \setminus X$ is closed (in the topology
  on~$\Q$ induced from the Euclidean topology on~$\R$).
\end{THM}

Finally, we consider the question of which groups can arise as the
automorphism group of a countable linearly ordered set; that is, which
groups can arise as the maximal subgroup of~$\End\Qset$.  In the case
of graphs, it is known by Frucht's Theorem~\cite{Frucht}, together
with the extension to infinite groups by de Groot~\cite{deGroot} and
Sabidussi~\cite{Sabidussi}, that every countable group arises as the
automorphism group of a countable graph.  In the case of linearly
ordered structures, it is easily determined from the definition that
there can be no non-trivial elements of finite order within the
automorphism group of a linearly ordered set.  On the other hand, such
automorphism groups can, of course, be rather complicated.  As an
example, $\Aut\Qset$~is an uncountable non-abelian group.  Although we
do not find a characterisation of all such groups (and it is unlikely
that one exists), we do show that the structure of a \emph{countable}
group that arises as the automorphism group of a countable linearly
ordered set is considerably constrained as the following theorem
indicates.

\begin{THM}
  \label{thm:countableautoms}
  Let $(\Omega,{\leq})$~be a countable totally ordered set and assume
  that $\Aut(\Omega,{\leq})$~is countable.  Then
  $\Aut(\Omega,{\leq})$~is a free abelian group of finite rank.
\end{THM}

The structure of this paper is as follows.  In
Section~\ref{sec:prelims}, we introduce all the basic notation and
terminology that we require when discussing linearly ordered sets.  We
also recall the semigroup theory that we use, including stating
results from our previous paper~\cite{Jay1} that we depend upon here.
Section~\ref{sec:EndQ} contains information about the structure
of~$\End\Qset$.  We prove Theorem~\ref{thm:main} and establish the
results concerning the $\scL$- and $\scR$\nbd classes in this monoid.
Sections~\ref{sec:image} and~\ref{sec:countable} are devoted to the
proofs of Theorems~\ref{thm:image} and~\ref{thm:countableautoms},
respectively.

\section{Preliminaries and notation}
\label{sec:prelims}

This section is split into two parts.  The first introduces the
terminology that we use relating to relational structures,
particularly as appears in the context of linearly ordered sets.  The
second part of the section is devoted to the semigroup theoretical
concepts that we use, in particular, the results from~\cite{Jay1} that
we depend upon.

\subsection{Relational structures}

A \emph{relational structure} is a pair $\Gamma = (V,\mathcal{E})$
consisting of a non-empty set~$V$ and a sequence $\mathcal{E} =
(E_{i})_{i \in I}$ of relations on~$V$.  All the examples in this
paper of relational structure will involve binary relations~$E_{i}$
and will mostly be \emph{linearly ordered sets}~$\oset{V}$, which is
where $\leq$~is a reflexive, transitive and anti-symmetric relation
on~$V$ such that for every pair $u,v \in V$ either $u \leq v$ or $v
\leq u$.  In view of this, throughout we shall refer only to binary
relations below although some definitions could be made in greater
generality.

If $\Gamma = (V,(E_{i})_{i \in I})$ and $\Delta = (W,(F_{i})_{i \in
  I})$ are relational structures (with relations indexed by the same
set~$I$), a \emph{homomorphism} $f \colon \Gamma \to \Delta$ is a map
$f \colon V \to W$ such that $(uf,vf) \in F_{i}$ whenever $(u,v) \in
E_{i}$.  In the case of linearly ordered sets, we shall also use the
term \emph{order-preserving map} as a synonym for homomorphism.  Thus,
if $\Gamma = \oset{V}$ and $\Delta = \oset{W}$ are linearly ordered
sets, a map $f \colon V \to W$ defines a order-preserving map $\Gamma
\to \Delta$ if $u \leq v$ in~$\Gamma$ implies $uf \leq vf$
in~$\Delta$.  Note here that we are following the convention of
writing maps on the right so that the image of a point~$v \in V$ under
a map is denoted by~$vf$.  The \emph{image} of a order-preserving map
$f \colon \oset{V} \to \oset{W}$ is the linearly ordered
set~$\im f = \oset{Vf}$ induced on the set of image values of~$f$.
The \emph{kernel} of~$f$ is the equivalence relation $\ker f =
\set{(u,v) \in V \times V}{uf = vf}$.  For each image value~$x$
of~$f$, the associated \emph{kernel class} is the preimage $xf^{-1} =
\set{v \in V}{vf = x}$ and this is, of course, one of the equivalence
classes defined by the kernel.

Furthermore, in our context an \emph{order-isomorphism} is an
isomorphism between two ordered sets.  The term \emph{order-embedding}
is used to refer to a map $f \colon \oset{V} \to \oset{W}$ between two
ordered sets such that $v \leq w$ in~$\oset{V}$ if and only if $vf
\leq wf$ in~$\oset{W}$.  It follows from this definition that $f$~is
injective, but one should note that in the case of \emph{linearly
  ordered} sets $\oset{V}$~and~$\oset{W}$ an injective
order-preserving map is always an order-embedding.

A linearly ordered set~$\oset{V}$ will be called \emph{dense} if for
every pair $u,v \in V$ with $u < v$ there exists some $w \in V$
satisfying $u < w < v$.  If $A$~is a subset of~$V$, a \emph{maximum
  element} of~$A$ is some $v \in A$ such that $a \leq v$ for all $a
\in A$.  The concept of \emph{minimum element} is defined dually.
  
If $\Gamma = \oset{V}$~is a linearly ordered set, a subset~$U$ of~$V$
is called \emph{convex} if whenever $u$,~$v$ and~$w$ are points in~$V$
with $u < w < v$ and $u,v \in U$ then necessarily $w \in U$ also.
Note that the kernel class~$xf^{-1}$ of a point in the image of a
order-preserving map $f \colon \oset{V} \to \oset{W}$ is always
convex.  Other examples of convex subsets of~$\Gamma$ include
\emph{intervals}.  We shall borrow the standard notation used for
intervals in the real line for intervals in~$\Gamma$:
\begin{align*}
  [u,v] &= \set{x \in V}{u \leq x \leq v} \\
  (u,v) &= \set{x \in V}{u < x < v}
\end{align*}
with $(u,v]$~and~$[u,v)$ being defined similarly for points $u,v \in
V$ with $u \leq v$.  A further point to observe is that if $A$~and~$B$
are disjoint convex subsets of~$\Gamma$, then it must be the case that
either $a < b$ for all $a \in A$ and~$b \in B$, or that $a > b$ for
all such $a$~and~$b$.  We shall write $A < B$ or $A > B$,
respectively, to indicate these situations.  Similarly, we shall use
the notation $a < B$ as a short-hand to mean that the element~$a$
satisfies $a < b$ for all $b \in B$.

The class of finite linearly ordered sets possesses the hereditary
property, the joint embedding property and the amalgamation property.
This class therefore has a unique Fra\"{\i}ss\'{e}
limit~\cite{Fraisse} (see, for example, \cite[Theorem~6.1.2]{Hodges}).
It is well-known that this Fra\"{\i}ss\'{e} limit is the ordered set
of rational numbers~$\Qset$.  Indeed, this structure is the unique
countable dense linearly ordered set with no maximum or minimum
elements, in the sense that any linearly ordered set satisfying these
properties is order-isomorphic to~$\Qset$.

At certain points, it will be helpful to view the set of ordered
rational numbers~$(\Q,{\leq})$ as a substructure of the \emph{extended
  real numbers}; that is, the relational structure $\R^{\ast} = ( \R
\cup \{ \pm\infty \}, {\leq})$ where $\leq$~denotes the usual order
on~$\R$ together with $-\infty \leq x \leq \infty$ for all $x \in \R$.
We shall then extend the above definitions of intervals to include
non-rational endpoints.  Thus, for example, if $p,q \in \R^{\ast}$, we
will be able to write $(p,q)$ for the interval \emph{in~$\Q$}
consisting of all \emph{rational} numbers~$x$ with $p < x < q$.  Note
by this notation we are always referring to a subset of~$\Q$ but that
the endpoints are permitted to be selected from outside of~$\Q$.  In
an attempt to avoid confusion, we shall never use the interval
notation to refer to subsets of~$\R$.

Note that if $p$~or~$q$ is not rational, then some of the intervals
$(p,q)$, $[p,q]$, $(p,q]$ and~$[p,q)$ coincide.  In addition, our
statement of Theorem~\ref{thm:image} above refers to specific
intervals in~$\Q$ as being closed in the topology induced on~$\Q$ from
the usual topology on~$\R$.  Note that an interval in~$\Q$ is closed
in this induced topology precisely when it \emph{can} be written in
the form~$[p,q]$ where $p,q \in \R \cup \{ \pm\infty \}$.  In
particular, the open interval~$(p,q)$ is closed if and only if $p,q
\in \R \setminus \Q$.  When establishing the theorem in
Section~\ref{sec:image} below, we shall make use of this description
of closed intervals in~$\Q$.

One of the properties of the relational structure~$\R^{\ast}$ is that
a non-empty subset~$A$ possesses a (possibly infinite) infimum and
supremum.  We shall write $\inf A$ and $\sup A$, as usual, to denote
these elements and we shall use these mostly for subsets~$A$ of~$\Q$,
but understand that in many cases  $\inf A$~and~$\sup A$ will be
non-rational elements of~$\R^{\ast}$.  Using these constructions,
though, one quickly observes that if $A$~is a convex subset of~$\Q$,
then $A$~equals one of the intervals $(p,q)$,~$[p,q)$, $(p,q]$
or~$[p,q]$ where $p = \inf A$ and $q = \sup A$ are some points
of~$\R^{\ast}$.

\subsection{Semigroup-theoretical notions}

Let $M = \End\Gamma$ be the endomorphism monoid of a linearly ordered
set $\Gamma = (V,\leq)$.  We say that two elements $f$~and~$g$ of~$M$
are $\scL$\nbd related if $f$~and~$g$ generate the same left ideal
(that is, $Mf = Mg$).  They are $\scR$\nbd related if $fM = gM$.
Green's $\scH$\nbd relation is the intersection of the binary
relations $\scL$~and~$\scR$, and the $\scD$\nbd relation is the
composite~$\scL \circ \scR$ (which is also always an equivalence
relation on~$M$).  The relevance of these relations to the study of
subgroups contained within the endomorphism monoid is that if $e$~is
an idempotent in~$M$ (that is, if $e^{2} = e$) then the $\scH$\nbd
class of~$e$ is a subgroup of~$M$ \ \cite[Corollary~2.2.6]{Howie} and,
morever, the collection of maximal subgroups of~$M$ are precisely the
$\scH$\nbd classes of idempotents of~$M$.  We use the term \emph{group
  $\scH$\nbd class} to refer to such maximal subgroups.

The following lemma is stated in greater generality for relational
structures in our previous paper~\cite{Jay1}.  The first two parts are
inherited from information concerning the $\scL$- and $\scR$\nbd
classes of the full transformation monoid~$\mathcal{T}_{V}$ of all
maps $V \to V$ (see \cite[Exercise~2.6.16]{Howie}).  Our restatement
here is simply interpreting that lemma in the context of linearly
ordered sets.

\begin{lemma}[\protect{\cite[Lemma~2.3]{Jay1}}]
  \label{lem:classes}
  Let $f$~and~$g$ be endomorphisms of the linearly ordered set $\Gamma
  = \oset{V}$.
  \begin{enumerate}
  \item If $f$~and~$g$ are $\scL$\nbd related, then $Vf = Vg$.
  \item If $f$~and~$g$ are $\scR$\nbd related, then $\ker f = \ker g$.
  \item If $f$~and~$g$ are $\scD$\nbd related, then $\im f$~and\/~$\im
    g$ are order-isomorphic.
  \end{enumerate}
\end{lemma}

An element~$f$ in the endomorphism monoid~$M$ is said to be
\emph{regular} if there exists $g \in M$ such that $fgf = f$.  An
idempotent endomorphism~$e$ is regular because $e^{3} = e$ and it is
known that if $f$~is regular, then every element in the $\scD$\nbd
class of~$f$ is also regular \cite[Proposition~2.3.1]{Howie}.  As
noted in~\cite{Jay1}, in the case of regular endomorphisms the
implications in Lemma~\ref{lem:classes} reverse.

\begin{lemma}[\protect{\cite[Lemma~2.5]{Jay1}}]
  \label{lem:reg-classes}
  Let $f$~and~$g$ be regular elements in the endomorphism monoid of
  the linearly ordered set $\Gamma = \oset{V}$.  Then
  \begin{enumerate}
  \item $f$~and~$g$ are $\scL$\nbd related if and only if\/ $Vf =
    Vg$.
  \item $f$~and~$g$ are $\scR$\nbd related if and only if\/ $\ker f =
    \ker g$.
  \item $f$~and~$g$ are $\scD$\nbd related if and only if\/ $\im
    f$~and\/~$\im g$ are order-isomorphic.
  \end{enumerate}
\end{lemma}

The final fact that we require from our earlier work is that we can
identify the group $\scH$\nbd classes in~$\End\Gamma$ as the
automorphism group of the image of endomorphisms within the class.

\begin{lemma}[\protect{\cite[Proposition~2.6(iii)]{Jay1}}]
  \label{lem:H-class}
  Let $f$~be an idempotent endomorphism of the linearly ordered
  set~$\Gamma = \oset{V}$.  Then the group $\scH$\nbd class~$H_{f}$ is
  isomorphic to the automorphism group of the image of~$f$.
\end{lemma}

\section{The structure of~\boldmath$\End(\pmb{\Q},\pmb{\leq})$}
\label{sec:EndQ}

This section is devoted to the study of the endomorphism monoid of the
ordered set~$\Qset$ of rational numbers.  The first stage is to prove
Theorem~\ref{thm:main} concerning the group $\scH$\nbd classes
within~$\End\Qset$.  In preparation for the proof, we establish
information concerning the number of idempotent endomorphisms with
specified image (Theorem~\ref{thm:idempotents}) and construct a family
of linearly ordered sets with trivial automorphism group
(Proposition~\ref{prop:uc-Cx}).  The remainder of the section is then
concerned with establishing our information about $\scL$- and
$\scR$\nbd classes and about non-regular $\scD$\nbd classes
in~$\End\Qset$.

\begin{thm}
  \label{thm:idempotents}
  Let\/ $\Omega = \oset{X}$~be any countable linearly ordered set.
  Then
  \begin{enumerate}
  \item if\/ $\card{X} = 1$, there are\/ $\aleph_{0}$~idempotent
    endomorphisms~$f$ of\/~$\Qset$ such that\/ $\im f \cong \Omega$;
  \item if\/ $\card{X} > 1$, there are\/ $2^{\aleph_{0}}$~idempotent
    endomorphisms~$f$ of\/~$\Qset$ such that\/ $\im f \cong \Omega$.
  \end{enumerate}
\end{thm}

\begin{proof}
  (i)~An endomorphism with image of cardinality~$1$ has the form $xf =
  q$ for all $x \in \Q$, where $q$~is some fixed point in~$\Q$.  All
  such endomorphisms are idempotent and there are countably many such
  maps.

  (ii)~Suppose $\card{X} > 1$.  By Theorem~\ref{thm:Kubis}, there
  exists an idempotent endomorphism~$f$ of~$\Qset$ with $\im f \cong
  \Omega$.  Choose some $q \in \im f$ subject to the condition that
  $q$~is not the maximum element of the image and let $I = qf^{-1}$.
  Put $\alpha = \inf I$ and $\beta = \sup I$, which are defined
  elements of the extended real numbers~$\R^{\ast}$.  By our
  assumption, $\beta < \infty$.  We shall define an idempotent
  endomorphism~$g_{\gamma}$ of~$\Qset$ as $g_{\gamma} = \xi f \eta$ in
  terms of certain $\xi,\eta \in \End\Qset$.  The definition of the
  latter will depend upon the choice of the parameter~$\gamma \geq
  \beta$, but will also be different according to whether or not
  $\beta \in \Q$ and whether $\beta$~lies in the interval~$I$.  The
  maps $\xi$~and~$\eta$ will be arranged so that $\eta\xi$~is the
  identity map and the image of~$g_{\gamma}$ is also isomorphic
  to~$\Omega$.

  \paragraph{Case~1: \boldmath$\beta \notin \Q$.}  In this case, we
  first choose any $\gamma \in \R \setminus \Q$ with with $\gamma \geq
  \beta$.  Note that there are uncountably many possible choices for
  such~$\gamma$.  Since $\beta \notin \Q$, necessarily $\alpha <
  \beta$, so we may also choose some $\delta \in \R \setminus \Q$
  satisfying $\alpha < \delta < \beta$.  The intervals
  $(\delta,\beta)$, $(\delta,\gamma)$, $(\beta,\infty)$
  and~$(\gamma,\infty)$ are all order-isomorphic to~$\Qset$ and so
  there are order-isomorphisms $\theta_{1} \colon (\delta,\beta) \to
  (\delta,\gamma)$ and $\theta_{2} \colon (\beta,\infty) \to
  (\gamma,\infty)$.  Define $\xi$ to be the order-automorphism
  of~$\Qset$ given by
  \[
  x\xi = \begin{cases}
    x &\text{for $x < \delta$} \\
    x\theta_{1} &\text{for $\delta < x < \beta$} \\
    x\theta_{2} &\text{for $x > \beta$}
  \end{cases}
  \]
  and $\eta$~to be its inverse.  Certainly then $\eta\xi$~is the
  identity map on~$\Q$ and $I\xi^{-1} = I \cup (\beta,\gamma)$ \ (that
  is, $I\xi^{-1}$~equals either~$(\alpha,\gamma)$
  or~$[\alpha,\gamma)$, depending upon whether or not $\alpha \in
    I$).

  \paragraph{Case~2: \boldmath$\beta \in \Q$ and $\beta \in I$.} In
  this case, consider any $\gamma \in \R$ with $\gamma \geq \beta$.
  We use an order-isomorphism~$\theta$ from~$(\gamma,\infty)$
  to~$(\beta,\infty)$, to define $\xi,\eta \in \End\Qset$ by
  \[
  x\xi = \begin{cases}
    x &\text{for $x \leq \beta$} \\
    \beta &\text{for $\beta < x \leq \gamma$} \\
    x\theta &\text{for $x > \gamma$}
  \end{cases}
  \]
  and
  \[
  x\eta = \begin{cases}
    x &\text{for $x \leq \beta$} \\
    x\theta^{-1} &\text{for $x > \beta$}.
  \end{cases}
  \]
  Then by construction, $\eta\xi$~is the identity map on~$\Q$ and
  $I\xi^{-1} = I \cup (\beta,\gamma]$ \ (that is, $I\xi^{-1}$~equals
  either~$(\alpha,\gamma]$ or~$[\alpha,\gamma]$, depending upon
  whether or not $\alpha \in I$).

  \paragraph{Case~3: \boldmath$\beta \in \Q$ but $\beta \notin I$.}
  Again we consider any $\gamma \in \R$ with $\gamma \geq \beta$.  We
  pick some $\delta \in \Q$ with $\delta > \gamma$ and let $\theta_{1}
  \colon (\alpha,\gamma) \to (\alpha,\beta)$ and $\theta_{2} \colon
  (\delta,\infty) \to (\beta,\infty)$ be order-isomorphisms.  Now
  define $\xi,\eta \in \End\Qset$ by
  \[
  x\xi = \begin{cases}
    x &\text{for $x \leq \alpha$} \\
    x\theta_{1} &\text{for $\alpha < x < \gamma$} \\
    \beta &\text{for $\gamma \leq x \leq \delta$} \\
    x\theta_{2} &\text{for $x > \delta$}
  \end{cases}
  \]
  and
  \[
  x\eta = \begin{cases}
    x &\text{for $x \leq \alpha$} \\
    x\theta_{1}^{-1} &\text{for $\alpha < x < \beta$} \\
    \delta &\text{for $x = \beta$} \\
    x\theta_{2}^{-1} &\text{for $x > \beta$}.
  \end{cases}
  \]
  Then $\eta\xi$~is the identity map on~$\Q$ and $I\xi^{-1} = I \cup
  [\beta,\gamma)$ \ (that is, $I\xi^{-1}$~equals
    either~$(\alpha,\gamma)$ or~$[\alpha,\gamma)$, depending upon
      whether or not $\alpha \in I$).

  \spc

  Using the $\xi$~and~$\eta$ just defined (depending upon, amongst
  other things, a choice of~$\gamma$), write~$g_{\gamma} = \xi f
  \eta$.  Using the fact that $f^{2} = f$ and $\eta\xi$~is the
  identity, we observe that $g_{\gamma}$~is also an idempotent
  endomorphism of~$\Qset$.  As $\xi$~is surjective and $\eta$~is
  injective in each case, it follows that $\im g_{\gamma} = \im f\eta
  \cong \im f \cong \Omega$.  Moreover, $qg_{\gamma}^{-1} = I\xi^{-1}$
  and this equals some interval~$J$ with $(\alpha,\gamma) \subseteq J
  \subseteq [\alpha,\gamma]$.

  In conclusion, as in each case  there are uncountably many choices
  for~$\gamma$, we have constructed $2^{\aleph_{0}}$~idempotent
  endomorphisms~$g_{\gamma}$ with image isomorphic to~$\Omega$.
\end{proof}

We shall now embark upon the proof of our main theorem.  The first
step is to construct some linearly ordered sets with trivial
automorphism group.

Consider an enumeration $\x = (x_{n})$ of the set~$\Q$ of rational
numbers.  Define a set~$C_{\x}$ depending upon this enumeration as a
set of ordered pairs of rational numbers and integers as follows:
\[
C_{\x} = \set{(x_{n},i)}{n \in \N, \; 0 \leq i \leq n}.
\]
We write $\C_{\x} = \oset{C_{\x}}$ for the linearly ordered set where
$\leq$~is the lexicographic order on~$C_{\x}$:
\[
(x_{m},i) \leq (x_{n},j)
\qquad \text{if and only if} \qquad
\text{either $x_{m} < x_{n}$, or both $m = n$ and $i \leq j$}.
\]
Essentially this definition arranges the points in the set $X_{n} =
\set{(x_{n},i)}{0 \leq i \leq n}$ in increasing order as indexed
by~$i$ and then orders the sets~$X_{n}$ relative to each other
according to the linear order on~$\Q$.  Thus, we are in effect
constructing~$\C_{\x}$ from~$\Qset$ by replacing each point~$x_{n}$
in~$\Q$ by a finite chain of length~$n$.

It is a straightforward observation, using the fact that $\Qset$~is
linearly ordered, to observe that $\C_{\x}$~is also a linearly ordered
set.  Furthermore, we similarly deduce the following facts.

\begin{lemma}
  \label{lem:Cx-basic}
  Consider two points $(x_{m},i), (x_{n},j) \in C_{\x}$.
  \begin{enumerate}
  \item There are infinitely many $c \in C_{\x}$ and infinitely many
    $d \in C_{\x}$ such that $c < (x_{m},i) < d$.
  \item If $x_{m} < x_{n}$, then there exist infinitely many $c \in
    C_{\x}$ such that $(x_{m},i) < c < (x_{n},j)$.
  \item For every~$i$ with $0 \leq i < n-1$, there exist no element~$c
    \in C_{\x}$ such that $(x_{n},i) < c < (x_{n},i+1)$.
  \end{enumerate}
\end{lemma}

\begin{lemma}
  \label{lem:Cx-map}
  Let $\x = (x_{n})$ and $\y = (y_{n})$ be two enumerations
  of\/~$\Q$.  Then $\C_{\x} = (C_{\x},\leq)$ and $\C_{\y} =
  (C_{\y},\leq)$ are order-isomorphic if and only if the map $x_{n}
  \mapsto y_{n}$, for $n \in \N$, defines an automorphism of the
  ordered set~$\Qset$.  Specifically, if $\phi$~is an
  order-isomorphism from~$\C_{\x}$ to~$\C_{\y}$, then $(x_{n},i)\phi =
  (y_{n},i)$ for all $n \in N$ and $0 \leq i \leq n$.
\end{lemma}

\begin{proof}
  Suppose $\phi$~is an order-isomorphism from~$\C_{\x}$ to~$\C_{\y}$.
  Then for $n \in \N$, write $X_{n} = \set{(x_{n},i)}{0 \leq i \leq
    n}$ and $Y_{n} = \set{(y_{n},i)}{0 \leq i \leq n}$.  We shall
  first observe that, for each $n \in \N$, there exists some~$m$ with
  $X_{n}\phi \subseteq Y_{m}$.  Suppose, for a contradiction, that
  $X_{n}\phi \nsubseteq Y_{m}$ for all $m \in \N$.  The
  elements~$(x_{n},i)$, for $0 \leq i \leq n$, are then mapped into at
  least two different sets~$Y_{m}$ and so there is some value~$k$ such
  that $(x_{n},k)\phi \in Y_{m}$ and $(x_{n},k+1)\phi \in Y_{m'}$ for
  distinct values $m,m' \in \N$.  Since $(x_{n},k)\phi <
  (x_{n},k+1)\phi$, it must be the case that $y_{m} < y_{m'}$.  Now by
  Lemma~\ref{lem:Cx-basic}(ii), there exists $c \in C_{\y}$ satisfying
  $(x_{n},k)\phi < c < (x_{n},k+1)\phi$.  We then conclude $(x_{n},k)
  < c\phi^{-1} < (x_{n},k+1)$, which contradicts
  Lemma~\ref{lem:Cx-basic}(iii).

  In conclusion, there exists some~$m$ such that $X_{n}\phi \subseteq
  Y_{m}$.  However, $\phi^{-1}$~also defines an order-isomorphism
  from~$\C_{\y}$ to~$\C_{\x}$ and the set~$Y_{m}\phi^{-1}$ contains
  all the points from~$X_{n}$.  The argument above applied
  to~$\phi^{-1}$ then establishes that $Y_{m}\phi^{-1} = X_{n}$ and so
  $X_{n}\phi = Y_{m}$.  Since $X_{n}$~contains precisely $n$~points,
  we conclude that $m = n$.  Since $\phi$~is order-preserving, it now
  follows that $(x_{n},i)\phi = (y_{n},i)$ for all $n \in \N$ and $0
  \leq i \leq n$, as claimed in the statement of the lemma.  Now, if
  $x_{m} \leq x_{n}$, then it follows that $(x_{m},0)\phi \leq
  (x_{n},0)\phi$, so necessarily $y_{m} \leq y_{n}$.  Hence we
  conclude that the map $x_{n} \mapsto y_{n}$ is indeed an
  automorphism of~$\Qset$.

  Conversely, if $x_{n} \mapsto y_{n}$ is an automorphism of~$\Qset$,
  then the map $(x_{n},i) \mapsto (y_{n},i)$, for $1 \leq i \leq n$
  and $n \in \N$, defines an order-isomorphism $\C_{\x} \to \C_{\y}$.
  This completes the proof of the lemma.
\end{proof}

Taking $\y = \x$ in the formula for order-isomorphisms in the previous
lemma yields:

\begin{cor}
  \label{cor:AutCx}
  Let $\x = (x_{n})$ be an enumeration of\/~$\Q$ and $\C_{\x} =
  \oset{C_{\x}}$.  Then $\Aut \C_{\x}$~is trivial. \qed
\end{cor}

\begin{prop}
  \label{prop:uc-Cx}
  There exists a set~$P$ of\/ $2^{\aleph_{0}}$~many enumerations of
  the set\/~$\Q$ of rational numbers such that $\C_{\mathbf{x}}
  \not\cong \C_{\mathbf{y}}$ for distinct $\mathbf{x},\mathbf{y} \in
  P$.
\end{prop}

\begin{proof}
  Fix one enumeration $\mathbf{x} = (x_{n})$ of the set~$\Q$.  For
  each~$i \in \N$, define~$\pi_{i}$ to be the transposition~$(2i \;\:
  2i+1)$ in the symmetric group~$\Sym(\N)$ and, for any subset~$A
  \subseteq \N$, define the involution $\pi_{A} = \prod_{i \in A}
  \pi_{i}$.  We shall write~$\x\pi_{A}$ for the
  enumeration~$(x_{n\pi_{A}})$ of~$\Q$ and set $P = \set{\x\pi_{A}}{A
    \subseteq \N}$.  Note that, for distinct subsets $A,B \subseteq
  \N$, the map given by $x_{n\pi_{A}} \mapsto x_{n\pi_{B}}$ for $n \in
  \N$ cannot be an order-automorphism of~$\Qset$, since
  $\pi_{A}\pi_{B}$~is again an involution.  It then follows that the ordered
  sets~$\C_{\x\pi_{A}}$, for $A \subseteq \N$, are pairwise
  non-isomorphic by Lemma~\ref{lem:Cx-map}.
\end{proof}

If $\Omega = (U,{\leq_{1}})$ and $\Lambda = (V,{\leq_{2}})$ are two
linearly ordered sets, we can define a new ordered set, that we shall
denote by~$\Omega + \Lambda$, as $\Omega + \Lambda = \oset{U \cup V}$,
where we assume that the sets $U$~and~$V$ are disjoint and where we
define the order~$\leq$ on~$U \cup V$ by $v \leq w$ if and only if one
of the following conditions holds (i)~$v,w \in U$ and $v \leq_{1} w$,
(ii)~$v,w \in V$ and $v \leq_{2} w$, or (iii)~$v \in U$ and $w \in V$.
In effect, in~$\Omega + \Lambda$, we are retaining the order in both
$\Omega$~and~$\Lambda$ but in addition are placing all points
in~$\Omega$ before all points in~$\Lambda$.  One observes immediately
that $\Omega + \Lambda$~is then also a linearly ordered set and is the
union of two substructures isomorphic to~$\Omega$ and to~$\Lambda$,
respectively.

\begin{prop}
  \label{prop:X+C}
  Let $\Omega = \oset{V}$ be any linearly ordered set.
  \begin{enumerate}
  \item If\/ $\x$~is any enumeration of the set\/~$\Q$ of rational
    numbers, then $\Aut ( \Omega + \C_{\x} )$~is isomorphic to~$\Aut
    \Omega$.
  \item If\/ $\x$~and~$\y$ are enumerations of\/~$\Q$, then $\Omega +
    \C_{\x}$~is order-isomorphic to~$\Omega + \C_{\y}$ if and only if
    $\C_{\x}$~is order-isomorphic to~$\C_{\y}$.
  \end{enumerate}
\end{prop}

\begin{proof}
  (i)~Recall that the set of points in~$\Omega + \C_{\x}$ is the
  union~$V \cup C_{\x}$.  To simplify notation, we shall write
  $\leq$~for the order both on~$\Omega$ and on~$\Omega + \C_{\x}$,
  since they coincide for points in the set~$V$.  We shall first,
  using a variation of the argument employed in
  Lemma~\ref{lem:Cx-map}, show that if $f \in \Aut ( \Omega + \C_{\x}
  )$ then $Vf = V$ and $C_{\x}f = C_{\x}$.  As before, we
  write~$X_{n}$ for the subset $\set{(x_{n},i)}{0 \leq i \leq n}$
  of~$C_{\x}$.

  \paragraph{Case~1:} We first consider the case when $C_{\x}f
  \subseteq C_{\x}$.  Then for each $n \in \N$, the set~$X_{n}$ is
  mapped into~$C_{\x}$ and the same argument as used in
  Lemma~\ref{lem:Cx-map} shows that there exists some~$m = m(n)$ such
  that $X_{n}f \subseteq X_{m}$.  Note $m \geq n$ due to the
  cardinality of the two sets involved.

  If $m > n$, there exists some $c \in X_{m} \setminus X_{n}f$.  Pick
  some $d \in X_{n}$.  Now either $df < c$ or $c < df$.  We consider
  the case when $c < df$.  Now $c = vf$ for some $v \in V \cup
  C_{\x}$.  Note that $v \notin X_{n}$ by assumption.  We use parts
  (i)~or~(ii) of Lemma~\ref{lem:Cx-basic}, depending upon whether $v
  \in V$ or $v \in C_{\x}$, to produce infinitely many elements $b \in
  C_{\x}$ satisfying $v < b < d$.  Then $vf < bf < df$, so $bf \in
  X_{n}$ also.  This is a contradiction since $X_{n}$~is finite.  When
  $df < c$, the argument is identical (though we know immediately from
  the order on~$\Omega + \C_{\x}$ that necessarily $c$~is the image of
  a point from~$C_{\x}$).  We conclude that $m(n) = n$ for all $n \in
  \N$ and so our map satisfies $X_{n}f = X_{n}$ for all $n \in \N$ and
  hence $C_{\x}f = C_{\x}$.  It then follows that $Vf = V$ also.

  \paragraph{Case~2:} Suppose that $C_{\x}f \not\subseteq C_{\x}$.
  Then there exists some $c \in C_{\x}$ such that $cf \in V$.  Then if
  $d$~is any point in~$C_{\x}$, it satisfies $cf < d$, so that $c <
  df^{-1}$.  Necessarily then $df^{-1} \in C_{\x}$ and we deduce $d
  \in C_{\x}f$.  We conclude that in this case $C_{\x}f \supseteq
  C_{\x}$.  We then apply the inverse of~$f$ and note that $f^{-1}$~is
  an automorphism of~$\Omega + \C_{\x}$ that satisfies $C_{\x}f^{-1}
  \subseteq C_{\x}$.  Case~1 tells us that $C_{\x}f^{-1} = C_{\x}$ and
  hence $C_{\x}f = C_{\x}$, contrary to our assumption.

  \spc

  We now know that $C_{\x}f = C_{\x}$ and $Vf = V$ for every
  automorphism~$f$ of~$\Omega + \C_{x}$ and it is a simple matter to
  conclude that
  \[
  \Aut ( \Omega + \C_{\x} ) \cong \Aut \Omega \times \Aut \C_{x} \cong
  \Aut \Omega,
  \]
  with use of Corollary~\ref{cor:AutCx}.

  (ii)~This is established similarly.  We observe that if $\phi$~is an
  isomorphism from $\Omega + \C_{\x}$ to~$\Omega + \C_{\y}$, then we
  show, using the argument just used in part~(i), that $V\phi = V$ and
  $C_{\x}\phi = C_{\y}$.  It then follows that $\phi$~induces an
  isomorphism $\C_{\x} \to \C_{\y}$.
\end{proof}

We are now able to prove our main theorem concerning the maximal
subgroups of $\End\Qset$ as stated in the Introduction.

\begin{proofof}{Theorem~\ref{thm:main}}
  (i)~Let $\Omega$~be any countable linearly ordered set.  Let $P$~be
  a set of $2^{\aleph_{0}}$~many enumerations of~$\Q$ such that
  $\C_{\x} \not\cong \C_{\y}$ when $\x$~and~$\y$ are distinct members
  of~$P$, as provided by Proposition~\ref{prop:uc-Cx}.  Now if $\x \in
  P$, then $\Omega + \C_{\x}$~is some countable linearly ordered set
  and so, by Theorem~\ref{thm:Kubis}, $\Omega + \C_{\x}$~is isomorphic
  to some retract of~$\Qset$; that is, there is an idempotent
  endomorphism~$f_{\x}$ of~$\Qset$ such that $\im f_{\x} \cong \Omega
  + \C_{\x}$.  Then the $\scH$\nbd class of~$f_{\x}$ is
  \[
  H_{f_{\x}} \cong \Aut ( \Omega + \C_{\x} ) \cong \Aut \Omega,
  \]
  by Proposition~\ref{prop:X+C}(i).  Hence each~$f_{\x}$ is an
  idempotent endomorphism with $\scH$\nbd class isomorphic to the
  automorphism group of~$\Omega$.

  Observe, moreover, that since $\Omega + \C_{\mathbf{x}} \not\cong
  \Omega + \C_{\mathbf{y}}$ for distinct $\mathbf{x},\mathbf{y} \in P$
  as shown in Proposition~\ref{prop:X+C}(ii), the $\scD$\nbd classes
  of the idempotent endomorphisms~$f_{\mathbf{x}}$ are distinct, using
  Lemma~\ref{lem:reg-classes}(iii).  Hence there are
  $2^{\aleph_{0}}$~distinct regular $\scD$\nbd classes of~$\End\Qset$
  with group $\scH$\nbd class isomorphic to~$\Aut\Omega$.

  (ii)~We make use of Theorem~\ref{thm:idempotents}.  Part~(i) of that
  theorem tells us that any endomorphism of~$\Qset$ with image of
  cardinality~$1$ is idempotent, therefore regular, and the set of all
  such endomorphisms forms a single $\scD$\nbd class~$D_{0}$ by
  Lemma~\ref{lem:reg-classes}(iii).  If $f \in D_{0}$, then $\{f\}$~is
  a single $\scH$\nbd class, again by use of
  Lemma~\ref{lem:reg-classes}, since any two endomorphisms in~$D_{0}$
  are $\scR$\nbd related but no distinct pair are $\scL$\nbd related.
  Thus $H_{f} = \{f\}$ and this is a copy of the trivial group.

  Suppose that $D$~is any other $\scD$\nbd class of~$\End\Qset$.  Fix
  $f_{0} \in D$ and write $\Omega = \im f_{0}$.  By
  Theorem~\ref{thm:idempotents}(ii), there are
  $2^{\aleph_{0}}$~idempotent endomorphisms~$f$ of~$\Qset$ with $\im f
  \cong \Omega$.  Each such~$f$ belongs to~$D$ by
  Lemma~\ref{lem:reg-classes}(iii) and determines a distinct group
  $\scH$\nbd class $H_{f} \cong \Aut\Omega$ by
  Lemma~\ref{lem:H-class}.  This completes the proof of the theorem.
\end{proofof}

The first paragraph of the proof of Theorem~\ref{thm:main}(ii) above
also establishes part~(i) of our result about the $\scR$\nbd classes
of~$\End\Qset$ as follows.

\begin{thm}
  \label{thm:R-classes}
  Let $f$~be an endomorphism of~$\Qset$ and write $X = \im f$.  Then
  \begin{enumerate}
  \item if\/ $\card{X} = 1$, the $\scD$\nbd class of~$f$ is a single
    $\scR$\nbd class;
  \item if\/ $\card{X} > 1$, the $\scD$\nbd class of~$f$ contains
    $2^{\aleph_{0}}$ many $\scR$\nbd classes.
  \end{enumerate}
\end{thm}

\begin{proof}
  (ii)~Assume that $\card{X} > 1$.  Our argument is similar to that
  which establishes part~(ii) of Theorem~\ref{thm:idempotents} above.
  Indeed, choose $q \in X$ that is not the maximum element of~$X$, put
  $I = qf^{-1}$, \ $\alpha = \inf I$ and $\beta = \sup I$.  Choose
  $\gamma$~to be a suitable real number with $\gamma \geq \beta$ and
  then define maps $\xi$~and~$\eta$ by the same formulae (depending
  upon whether $\beta \in \Q$ and whether $\beta \in I$) as found in
  the proof of Theorem~\ref{thm:idempotents}.  Then, as noted before,
  $\eta\xi$~is the identity map on~$\Q$ and $(\alpha,\gamma) \subseteq
  I\xi^{-1} \subseteq [\alpha,\gamma]$.

  Now the map $\xi f$~is $\scL$\nbd related to~$f$ in view of the
  formula $\eta\xi f = f$.  Note that the kernel class of all points
  in~$\Q$ that map to~$q$ under~$\xi f$ equals $q(\xi f)^{-1} =
  I\xi^{-1}$, whose form is as described above.  Thus as
  $\gamma$~varies, we obtain $2^{\aleph_{0}}$~endomorphisms in the
  $\scD$\nbd class of~$f$ that are not $\scR$\nbd related to each
  other because they have distinct kernels.
\end{proof}

\begin{prop}
  \label{prop:finite->regular}
  Let $f$~be any endomorphism of~$\Qset$.  If\/ $\im f$~is finite,
  then $f$~is regular.
\end{prop}

\begin{proof}
  Let $X = \im f$ and $x_{1}$,~$x_{2}$, \dots,~$x_{n}$ be the distinct
  image values of~$f$.  Choose $q_{i} \in x_{i}f^{-1}$ for each~$i$.
  There is an automorphism~$g$ of~$\Qset$ satisfying $x_{i}g = q_{i}$
  for each~$i$.  Then $g \in \End\Qset$ and $fgf = f$.  Hence $f$~is
  regular.
\end{proof}

\begin{thm}
  \label{thm:L-classes}
  Let $f$~be an endomorphism of~$\Qset$ and write $X = \im f$.  Then
  \begin{enumerate}
  \item if $X$~is finite, the $\scD$\nbd class of~$f$ contains
    $\aleph_{0}$~many $\scL$\nbd classes;
  \item if $X$~is infinite, the $\scD$\nbd class of~$f$ contains
    $2^{\aleph_{0}}$~many $\scL$\nbd classes.
  \end{enumerate}
\end{thm}

\begin{proof}
  (i)~First note that if $X$~is finite, then $f$~is regular by
  Proposition~\ref{prop:finite->regular} and, indeed, by
  Lemma~\ref{lem:reg-classes}(iii), another endomorphism~$g$ is
  $\scD$\nbd related to~$f$ if and only if $\card{\im g} = \card{X}$.
  Given two endomorphisms $g$~and~$h$ with images of the same
  cardinality, they are $\scL$\nbd related if and only if their images
  are equal (in addition to being isomorphic).  There are countably
  many choices for a subset of~$\Q$ of a particular finite cardinality
  and hence the $\scD$\nbd class of~$f$ contains countably many
  $\scL$\nbd classes.

  (ii)~Now suppose that $X$~is infinite.  We divide into two cases:
  \begin{enumerate}
    \renewcommand{\labelenumi}{(\alph{enumi})}
  \item Either $X$~contains an infinite sequence~$(x_{n})$ of points
    such that, for each~$n$, \ $x_{n}$~is the maximum member of~$X
    \setminus \{ x_{1}, x_{2}, \dots, x_{n-1} \}$, or
  \item there are finitely many points $x_{1}$,~$x_{2}$,
    \dots,~$x_{n}$ in~$X$ such that $x_{i}$~is the maximum member
    of~$X \setminus \{ x_{1}, x_{2}, \dots, x_{i-1} \}$ for $i \leq
    n$, and such that $X \setminus \{ x_{1},x_{2},\dots,x_{n} \}$~has
    no maximum member.
  \end{enumerate}
  (When $X$~has no maximum element, we are in Case~(b) with $n = 0$.)

  Suppose then that we are in Case~(a).  Put $Y = X \setminus \{
  x_{1},x_{2},\dots \}$ and let $\alpha = \sup Y$.  If $Y$~is empty,
  take $\alpha = -\infty$.  Note then that $\alpha < x_{n} < x_{n-1}$
  for all~$n$.  Now pick any rational number $q_{1} > \alpha$
  and, having chosen $q_{1}$,~$q_{2}$, \dots,~$q_{n-1}$, pick any
  rational number~$q_{n}$ satisfying $\alpha < q_{n} < q_{n-1}$.
  There are $2^{\aleph_{0}}$~many ways of choosing the resulting
  sequence $\mathbf{q} = (q_{n})$.  Now $(\alpha,\infty)$~is order
  isomorphic to~$\Q$ and hence there is an order-preserving
  bijection~$\xi = \xi_{\mathbf{q}}$ from~$(\alpha,\infty)$ to itself
  that maps~$x_{n}$ to~$q_{n}$ for each~$n \in \N$.  Extend this to an
  automorphism~$\xi$ of~$\Qset$ by defining $x\xi = x$ for all $x \leq
  \alpha$.  Then $f\xi$~is $\scR$\nbd related to~$f$ and $\im f\xi =
  X\xi = Y \cup \{ q_{1},q_{2},\dots \}$.  Consequently, by
  Lemma~\ref{lem:classes}(i), all such~$f\xi$ lie in different
  $\scL$\nbd classes and we have established the claimed result in
  this case.

  We now turn to Case~(b).  Put $Y = X \setminus \{
  x_{1},x_{2},\dots,x_{n} \}$ and let $\alpha = \sup Y$.  (In this
  case, necessarily $Y$~is non-empty.)  Note, by assumption, $\alpha
  \notin X$.  Choose any real number~$\beta$ with $\beta < \alpha$.
  Since the intervals $(-\infty,\alpha)$ and $(-\infty,\beta)$ are
  order-isomorphic, there is an order-isomorphism $\theta \colon
  (-\infty,\alpha) \to (-\infty,\beta)$.  Pick any rational
  number~$\gamma$ with $\alpha \leq \gamma \leq x_{n}$.  Then define
  $\xi,\eta \in \End\Qset$ by
  \[
  x\xi = \begin{cases}
    x\theta &\text{if $x < \alpha$} \\
    x &\text{if $x \geq \alpha$}
  \end{cases}
  \]
  and
  \[
  x\eta = \begin{cases}
    x\theta^{-1} &\text{if $x < \beta$} \\
    \gamma &\text{if $\beta \leq x \leq \gamma$} \\
    x &\text{if $x \geq \gamma$}.
  \end{cases}
  \]
  Then $x_{i}\xi\eta = x_{i}$ for $i = 1$,~$2$, \dots,~$n$, since
  each~$x_{i} \geq \gamma$, while $\xi\eta$~is the identity map
  on~$Y$.  We therefore conclude $f\xi\eta = f$.  It follows that
  $f\xi$~and~$f$ are $\scR$\nbd related.  Moreover, $\im f\xi = X\xi =
  Y\theta \cup \{ x_{1},x_{2},\dots,x_{n} \}$ and $\sup Y\theta =
  \beta$.  Hence, as $\beta$~is permitted to vary through $\set{\beta
    \in \R}{\beta < \alpha}$, we obtain $2^{\aleph_{0}}$~endomorphisms
  in the $\scD$\nbd class of~$f$, all of which belong to distinct
  $\scL$\nbd classes by Lemma~\ref{lem:classes}(i).  This completes
  the proof.
\end{proof}

The following result provides a condition that is sufficient for
producing non-regular endomorphisms.  It is phrased in terms of the
infimum and supremum of a subset of~$\Q$.  We remind the reader that
these are well-defined members of~$\R^{\ast}$ and might not
necessarily be rational numbers in general.

\begin{thm}
  \label{thm:nonregular}
  Let $X$~be a subset of\/~$\Q$ with the property that $X$~has a
  partition into two disjoint subsets $X = X_{-} \cup X_{+}$ where
  $X_{-} < X_{+}$ and such that $\alpha = \sup X_{-}$ and $\beta =
  \inf X_{+}$ do not belong to~$X$.  Then there exists a non-regular
  endomorphism~$f$ of\/~$\Qset$ such that the image of~$f$ is
  order-isomorphic to the substructure~$\oset{X}$.
\end{thm}

\begin{proof}
  We make a number of reductions.  The first is to observe that we can
  assume that $X$~is the image of an endomorphism of~$\Qset$.  Indeed,
  by Theorem~\ref{thm:Kubis}, there is an (idempotent)
  endomorphism~$g$ of~$\Qset$ with image isomorphic to~$X$.  Write $Y
  = \im g$ and denote by $\phi$~the order-isomorphism from~$\oset{X}$
  to~$\oset{Y}$.  Let $Y_{-} = X_{-}\phi$ and $Y_{+} = X_{+}\phi$.
  Put $\gamma = \sup Y_{-}$.  If it were the case that $\gamma \in Y$,
  then $\gamma = x\phi$ for some $x \in X$.  This element~$x$ cannot
  be a member of~$X_{-}$, since $x$~would then be the maximum element
  of~$X_{-}$, contradicting the assuming that $\sup X_{-} \notin X$.
  Consequently, $x \in X_{+}$ and by assumption $\inf X_{+} < x$.  In
  particular, there exists some $y \in X_{+}$ with $\inf X_{+} < y <
  x$.  Then $y\phi$~is a point in~$Y_{+}$ satisfying $z < y\phi <
  \gamma$ for all $z \in Y_{-}$, which contradicts the definition
  of~$\gamma$ as the supremum of~$Y_{-}$.  We conclude, by symmetry,
  that neither $\sup Y_{-}$~nor~$\inf Y_{+}$ belong to~$Y$.  In
  conclusion, we can now replace~$X$ by~$Y$ and hence assume that
  $X$~is the image of the endomorphism~$g$.

  Our second reduction is to show that we can assume $\alpha = \beta$.
  Indeed, there is an order-isomorphism $\theta \colon (\beta,\infty)
  \to (\alpha,\infty)$ and we can define a new endomorphism~$g'$
  of~$\Qset$ by
  \[
  xg' = \begin{cases}
    xg &\text{if $xg \in X_{-}$} \\
    xg\theta &\text{if $xg \in X_{+}$}.
  \end{cases}
  \]
  The image of~$g'$ is order-isomorphic to~$X$ and is the disjoint
  union of~$X_{-}$ and~$X_{+}\theta$.  The infimum of~$X_{+}\theta$ is
  also~$\alpha$.  We may therefore replace~$g$ by the
  endomorphism~$g'$ and hence assume that $\alpha = \beta$.

  In summary, there is an endomorphism~$g$ of~$\Qset$ such that the
  image $\im g = X$ is a disjoint union $X = X_{-} \cup X_{+}$ with
  $X_{-} < X_{+}$ and $\sup X_{-} = \inf X_{+} = \alpha \notin X$.
  Pick any real number~$\delta < \alpha$.  There is an
  order-isomorphism $\xi \colon (-\infty,\alpha) \to (-\infty,\delta)$
  and we extend this to an endomorphism of~$\Qset$ by also defining
  $x\xi = x$ for all $x \geq \alpha$.  As $\xi$~is an order-embedding,
  we conclude that $g\xi \in \End\Qset$ and $\im (g\xi) = X\xi \cong
  X$.  We shall show that $f = g\xi$~is not regular.

  Suppose that $h$~is an endomorphism of~$\Qset$ with the property
  that $fhf = f$.  Since $\alpha = \sup X_{-} = \inf X_{+}$, there
  exist sequences $(x_{i})$~and~$(y_{i})$ in $X_{-}$~and~$X_{+}$,
  respectively, converging to~$\alpha$.  As $\xi$~is an
  order-isomorphism from~$(-\infty,\alpha)$ to~$(-\infty,\delta)$, we
  conclude that the sequence~$(x_{i}\xi)$ converges to~$\delta$.  Pick
  $q \in \Q$ with $\delta < q < \alpha$.  There is a
  sequence~$(q_{i})$ in~$\Q$ with $q_{i}g = x_{i}$ for each~$i$.  Now
  $q_{i}f = x_{i}\xi < \delta < q$, so $x_{i}\xi = q_{i}f = q_{i}fhf
  \leq qhf$ for each~$i$.  As $(x_{i}\xi)$~converges to~$\delta$, we
  conclude that $qhg\xi = qhf \geq \delta$.  From the definition
  of~$\xi$ and the fact that $\alpha \notin \im g = X$, we conclude
  $qhg > \alpha$.

  Similarly, there is a sequence~$(r_{i})$ in~$\Q$ with $r_{i}g =
  y_{i}$ for each~$i$.  Now $r_{i}f = y_{i}\xi > \delta > q$, so
  $y_{i} = y_{i}\xi = r_{i}f = r_{i}fhf \geq qhf$ for each~$i$.  The
  convergence of~$(y_{i})$ to~$\alpha$, allows us to conclude $qhf
  \leq \alpha$.  The definition of~$\xi$ forces $qhg < \alpha$.

  Comparing the conclusions of the last two paragraphs, we now have a
  contradiction and hence have established that $f$~is indeed not
  regular.
\end{proof}

Observe that if $X$~is a subset of~$\Q$ containing a dense interval,
then it satisfies the hypotheses of Theorem~\ref{thm:nonregular} since
we can choose an irrational number~$\alpha$ in the corresponding real
interval and then partition~$X$ into $X_{-} = \set{x \in X}{x <
  \alpha}$ and $X_{+} = \set{x \in X}{x > \alpha}$.  On the other
hand, if $X \cong (\N,\leq)$, then by a similar argument to
Proposition~\ref{prop:finite->regular} any endomorphism~$f$ with $\im
f \cong \N$ is regular.

\begin{cor}
  \label{cor:uc-nonreg}
  There are $2^{\aleph_{0}}$~non-regular $\scD$\nbd classes
  in~$\End\Qset$.
\end{cor}

\begin{proof}
  We again make use of the ordered sets~$\C_{\x}$ constructed
  earlier.  Let $P$~be the set of enumerations of~$\Q$ provided by
  Proposition~\ref{prop:uc-Cx}.  If $\x \in P$, it is possible to
  embed a copy of~$\C_{\x}$ as a subset~$D_{\x}$ of~$(2,\infty)$, as
  $(2,\infty)$~is order-isomorphic to~$\Q$.  Then take $X_{\x} = (0,1)
  \cup D_{\x}$.  As this set contains an interval, it satisfies the
  hypotheses of Theorem~\ref{thm:nonregular} with, for example,
  $\alpha = \beta = 1/\sqrt{2}$ and so there exists a non-regular
  endomorphism~$f_{\x}$ of~$\Qset$ with image isomorphic to~$X_{\x}$.

  Now if $\x$~and~$\y$ are distinct enumerations in~$P$, then $X_{\x}
  \not\cong X_{\y}$ by Proposition~\ref{prop:X+C} combined with the
  property of~$P$.  Hence $f_{\x}$~and~$f_{\y}$ are not $\scD$\nbd
  related by Lemma~\ref{lem:classes}(iii).  Thus we do indeed have
  $2^{\aleph_{0}}$~non-regular $\scD$\nbd classes of endomorphisms
  of~$\Qset$.
\end{proof}

\section{Images of idempotent transformations
  (Theorem~\ref{thm:image})}
\label{sec:image}

We shall now establish Theorem~\ref{thm:image}, namely that a
subset~$X$ of~$\Q$ arises as the image of an idempotent endomorphism
of~$\Qset$ if and only if no maximal interval within the complement
of~$X$ is closed.

\spc

First let $f$~be an idempotent endomorphism of the linearly ordered
set of rational numbers~$(\Q,{\leq})$.  In order to describe the image
of~$f$ as a subset of~$\Q$ we shall consider the various
preimages~$xf^{-1}$ of $x \in \Q$.  Note that $xf^{-1}$~is empty if
$x$~is not in the image of~$f$, while $x \in xf^{-1}$ for all $x \in
\im f$ because $f$~is idempotent.  Define
\[
J = \set{x \in \im f}{\card{xf^{-1}} > 1}.
\]
When $f$~is the identity map, $J = \emptyset$ and $\im f = \Q$.  For
all other idempotent endomorphisms~$f$, \ $J$~is non-empty and $\im
f$~is a proper subset of~$\Q$.  For the following analysis, we shall
assume $f$~is not the identity.

For each $x \in J$, we shall define below two (possibly empty)
intervals $L_{x}$~and~$U_{x}$ in~$\Q$.  The definition will depend
upon the infimum and supremum of the preimage set~$xf^{-1}$.  If
$\inf(xf^{-1}) \neq -\infty$, then the set $\set{q \in \im f}{q < x}$
is non-empty.  When this set is non-empty \emph{and} has a maximum
member, we shall define
\[
m_{x} = \max \set{q \in \im f}{q < x}.
\]
Dually, we define
\[
n_{x} = \min \set{q \in \im f}{q > x}
\]
when this minimum element exists.  These two values, when they exist,
will contribute to the definition of the intervals
$L_{x}$~and~$U_{x}$, as follows:

\begin{enumerate}
\item If $xf^{-1}$~is not bounded below, then we set $L_{x} =
  (-\infty,x)$.
\item If $\inf(xf^{-1}) \in \Q$ is an image value of~$f$, then we set
  $L_{x} = (\inf(xf^{-1}),x)$.  (This interval is empty in the case
  when $\inf(xf^{-1}) = x$.)
\item If $\inf(xf^{-1}) \in \R \setminus (\im f)$ and $m_{x} = \max
  \set{ q \in \im f }{ q < x }$ exists, then set $L_{x} = (m_{x},x)$.
\item Otherwise, set $L_{x} = [ \inf(xf^{-1}), x )$.
\end{enumerate}

We make a dual set of definitions for~$U_{x}$:
\begin{enumerate}
\item If $xf^{-1}$~is unbounded above, then we set $U_{x} = (x,+\infty)$.
\item If $\sup(xf^{-1}) \in \Q$ is an image value of~$f$, then we set
  $U_{x} = (x,\sup(xf^{-1}))$.  (This interval is empty in the case
  when $\sup(xf^{-1}) = x$.)
\item If $\sup(xf^{-1}) \in \R \setminus (\im f)$ and $n_{x} = \min
  \set{q \in \im f}{q > x}$ exists, then set $U_{x} = (x,n_{x})$.
\item Otherwise, set $U_{x} = (x,\sup(xf^{-1})]$.
\end{enumerate}

\begin{lemma}
  \label{lem:LU-nonclosed}
  Let $x \in \im f$ be such that\/ $\card{xf^{-1}} > 1$.  The
  intervals $L_{x}$~and~$U_{x}$ are either empty or are intervals
  in~$\Q$ that are not closed and are disjoint from the image of~$f$.
\end{lemma}

\begin{proof}
  We consider the interval~$L_{x}$ in the case when it is non-empty,
  since the argument for~$U_{x}$ is analogous.  Since one endpoint is
  $x \in \Q$ and $x \in \Q \setminus L_{x}$, we see that $L_{x}$~is
  not closed.  We shall now show that it cannot contain a point in the
  image of~$f$.

  As $f$~is order-preserving, we know that $qf = x$ for all~$q \in
  (\inf(xf^{-1}),x)$.  In particular, no point in~$(\inf(xf^{-1}),x)$
  lies in the image of the idempotent map~$f$.  In particular, this
  tells us that $L_{x}$~does not meet the image of~$f$ in Cases~(i),
  (ii) or~(iv) of its definition.

  Finally, in Case~(iii), $m_{x}$~is the maximum element of the image
  of~$f$ satisfying $m_{x} < x$.  Consequently, $L_{x} =
  (m_{x},x)$~again does not meet~$\im f$.
\end{proof}

\begin{lemma}
  \label{lem:LU-maxintervals}
  Let $x \in \im f$ be such that\/ $\card{xf^{-1}} > 1$.  If
    $L_{x}$~or~$U_{x}$ is non-empty, then it is a maximal interval
    within $\Q \setminus ( \im f )$.
\end{lemma}

\begin{proof}
  We deal with~$L_{x}$ and consider each of the Cases~(i)--(iv) above
  in the definition of this interval.  The result for~$U_{x}$ is
  established by a dual argument.  First note that, by
  Lemma~\ref{lem:LU-nonclosed}, $L_{x} \subseteq \Q \setminus ( \im f
  )$.  We shall show that $L_{x}$~is a maximal interval in this
  complement.

  In Case~(i), \ $L_{x} = (-\infty,x)$ and the endpoint~$x$ belongs
  to the image of~$f$.  In Case~(ii), \ $L_{x} = (\inf(xf^{-1}),x)$
  and the endpoints $\inf(xf^{-1})$~and~$x$ both belong to~$\im f$
  (the former by assumption).  In Case~(iii), \ $L_{x} = (m_{x},x)$
  and the endpoints $m_{x}$~and~$x$ both belong to~$\im f$.  Hence, in
  these three cases, $L_{x}$~is a maximal interval within~$\Q
  \setminus ( \im f )$.

  In Case~(iv), \ $L_{x} = [\inf(xf^{-1}),x)$ where $\inf(xf^{-1})
  \notin \im f$ and $M = \set{q \in \im f}{q < x}$ has no maximum
  element.  Note that $\inf(xf^{-1}) < x$ since the former is not an
  image value of~$f$ in this case.  Suppose there exists some
  interval~$I$ contained within $\Q \setminus (\im f)$ that strictly
  contains~$L_{x}$.  Note that the endpoint~$x$ of~$L_{x}$ belongs to
  the image of~$f$, so such~$I$ must contain some $r \in \Q$ with $r <
  \inf (xf^{-1})$.  It then follows that $rf < x$; that is, $rf \in
  M$.  As $M$~has no maximum element, we conclude that some $s \in \im
  f$ satisfies $rf < s < x$.  This element~$s$ does not belong to the
  interval~$I$ and hence $rf < s < r$.  Then $s = sf \leq rf$, which
  is a contradiction.  Hence no such interval~$I$ exists and we
  conclude that $L_{x}$~is indeed a maximal interval within $\Q
  \setminus (\im f)$ in this final case.  This completes the proof.
\end{proof}

\begin{proofof}{Theorem~\ref{thm:image}}
  Let $f$~be any idempotent endomorphism of~$(\Q,{\leq})$, let
  \[
  J = \set{x \in \im f}{\card{xf^{-1}} > 1},
  \]
  and for each~$x \in J$, define the sets $L_{x}$~and~$U_{x}$ as
  described above.  Our first step will be to observe that every point
  in the complement of the image of~$f$ belongs to at least one of the
  sets $L_{x}$~or~$U_{x}$ for some $x \in J$.

  Let $q \in \Q \setminus ( \im f )$ and put $x = qf$.  Then
  $xf^{-1}$~contains~$q$, so by definition $x \in J$.  Now either $q <
  x$ or $q > x$.  We shall consider the case when $q < x$.  The
  definition tells us that $\inf(xf^{-1}) \leq q$.  We now analyse the
  definition of~$L_{x}$ and split into Cases~(i)--(iv) as above.  In
  Case~(i), $\inf(xf^{-1}) = -\infty < q$ and so $q \in L_{x} =
  (-\infty,x)$.  In Case (ii), the assumption that $\im f$~is an image
  value ensures $\inf(xf^{-1}) < q$, so $q \in L_{x} =
  (\inf(xf^{-1}),x)$.  In Case~(iii), the facts that $f$~is an
  order-preserving idempotent, $qf = x$ and $m_{x}$~is the maximum
  image value satisfying $m_{x} < x$ implies that $m_{x} < q$ and so
  $q \in L_{x} = (m_{x},x)$.  Finally, in Case~(iv), we already know
  that $\inf(xf^{-1}) \leq q$, so $q \in L_{x} = [\inf(xf^{-1}),x)$.
  Similarly, if $q > x$, then $q \in U_{x}$.  In conclusion, every
  point~$q$ not in the image of~$f$ lies in either $L_{x}$~or~$U_{x}$
  with~$x = qf \in J$.

  We have already observed, in Lemma~\ref{lem:LU-maxintervals}, that
  the sets $L_{x}$~and~$U_{x}$ are maximal intervals in~$\Q \setminus
  ( \im f )$ and it now follows, from the previous paragraph, that
  these sets are \emph{all} the maximal intervals in~$\Q \setminus (
  \im f )$.  We have observed that these sets are not closed in
  Lemma~\ref{lem:LU-nonclosed}.  This establishes the necessity part
  of Theorem~\ref{thm:image}.

  \spc

  Conversely, suppose that $X$~is a subset of~$\Q$ and that $\Q
  \setminus X = \bigcup_{i \in I} T_{i}$, where the sets~$T_{i}$, for
  $i \in I$, are the maximal intervals in~$\Q \setminus X$.  Assume
  that the~$T_{i}$ are not closed.  We define a map $f \colon \Q \to
  \Q$ as follows.

  Consider one of the intervals~$T_{i}$.  Since it cannot be expressed
  as a closed interval with endpoints $q,r \in \R \cup
  \{\pm\infty\}$, it has one of the following forms:
  \begin{enumerate}
  \item $T_{i} = [q,r)$ for some $q$~and~$r$ with necessarily $r \in
    \Q$.  In this case, define $xf = r$ for all $x \in [q,r)$.
    \item $T_{i} = (q,r]$ for some $q$~and~$r$ with necessarily $q \in
    \Q$.  In this case, define $xf = q$ for all $x \in (q,r]$.
  \item $T_{i} = (q,r)$ for some $q$~and~$r$.  Note that at least one
    of $q$~or~$r$ is rational, since otherwise we could write $T_{i} =
    [q,r]$ contrary to the assumption that $T_{i}$~is not a closed
    interval.  We then define $f$~on this interval depending upon
    which endpoint is rational:
    \[
    xf = \begin{cases}
      q &\text{when $q \in \Q$} \\
      r &\text{when $q \notin \Q$}
    \end{cases}
    \]
    for all $x \in (q,r)$.
\end{enumerate}
Finally define $xf = x$ for all $x \in X$.  In this way, we have
defined~$f$ on the whole set~$\Q$.  To verify that $f$~is an
idempotent endomorphism of~$\Qset$ with image equal to~$X$, we now
proceed as follows.

First, if $T_{i} = [q,r)$~is a maximal interval in~$\Q \setminus X$
with $r \in \Q$, then $r$~cannot belong to another maximal
interval~$T_{j}$ (as otherwise $T_{i} \cup T_{j}$ would be a larger
interval in~$\Q \setminus X$).  Hence $r$~belongs to the set~$X$.
Similar arguments apply to the other cases in the definition of~$f$,
so we conclude that $\im f = X$.  As a consequence, since $xf = x$ for
all $x \in X$, it now follows that $f$~is idempotent.

Finally, we observe that $f$~is an endomorphism of~$\Qset$.  Let $x,y
\in \Q$ satisfy $x < y$.  When $x,y \in X$, there is nothing to
establish since $xf = x$ and $yf = y$.  Suppose that $x \in T_{i}$ for
some~$i$ and that $y \in X$.  Let the endpoints of~$T_{i}$ be
$q$~and~$r$ with $q < r$.  Then necessarily $q < r \leq y$.  Our
definition for~$f$ states that $xf$~equals one of $q$~or~$r$.  Either
way, we know $xf \leq r \leq y = yf$.  A similar argument applies when
$x \in X$ and $y \in T_{i}$ for some~$i$.

The remaining case is when both $x$~and~$y$ lie in one of the maximal
intervals~$T_{i}$.  If they lie in the \emph{same} maximal interval,
then $xf = yf$.  If, say, $x \in T_{i}$ and $y \in T_{j}$ with $i \neq
j$, let the endpoints of~$T_{i}$ and of~$T_{j}$ be $q_{1},r_{1}$ and
$q_{2},r_{2}$, respectively.  Then $q_{1} < r_{1} \leq q_{2} <
r_{2}$.  The definition of~$f$ tells us $xf \in \{ q_{1},r_{1} \}$ and
$yf \in \{ q_{2},r_{2} \}$ and $xf \leq yf$ follows.  Hence $f$~is
indeed an idempotent endomorphism of~$\Qset$ with image equal to the
set~$X$.

This completes the proof of Theorem~\ref{thm:image}.
\end{proofof}

\section{Countable automorphism groups of countable linearly ordered
  structures (Theorem~\ref{thm:countableautoms})}
\label{sec:countable}

Let $\Omega = (V,\leq)$ where $V$~is a countable set and $\leq$~is a
linear order on~$V$.  Throughout this section, we assume that
$\Aut\Omega$~is a countable group.  Our goal in this section is to
show that this group is free abelian of finite rank.

Observe that if $X$~is a convex subset of~$V$, then every
automorphism~$\phi$ of~$\oset{X}$ can be extended to an automorphism
of~$\Omega$ by defining
\[
v\hat{\phi} = \begin{cases}
  v\phi &\text{for $v \in X$} \\
  v &\text{for $v \in V \setminus X$.}
\end{cases}
\]
Thus $\Aut X$~embeds as a subgroup of~$\Aut\Omega$ and so our
assumption implies that $\Aut X$~is countable for every convex
subset~$X$ of~$V$.  We shall use this and similar ideas throughout our
argument in this section.

If $f \in \Aut\Omega$, define a relation~$\sim$ on~$V$ by $x \sim y$
if and only if $xf^{m} \leq y \leq xf^{n}$ for some $m,n \in \Z$.
(Note that $\sim$~depends upon the automorphism~$f$, but for
simplicity of notation we choose not to write~$\sim_{f}$ for this
relation.)  Then $\sim$~is an equivalence relation on~$V$ and we
define the \emph{orbital}~$U_{f}(x)$ (following Truss~\cite{Truss})
to be the equivalence class of the point~$x$ under the
relation~$\sim$.  Observe that if $x$~is fixed by~$f$, then $U_{f}(x)
= \{ x \}$, while if $xf \neq x$ then the values~$xf^{n}$, as
$n$~ranges through~$\Z$, are distinct and it then follows from the
definition that $U_{f}(x)$~is an infinite convex subset of~$V$.

The following contains the basic properties of orbitals that we shall
need.

\begin{lemma}
  \label{lem:orbital}
  Let $f,g \in \Aut\Omega$ and $x \in V$.
  \begin{enumerate}
  \item If $xf > x$, then $U_{f}(x)$~is infinite and $uf > u$ for all
    $u \in U_{f}(x)$.
  \item If $xf < x$, then $U_{f}(x)$~is infinite and $uf < u$ for all
    $u \in U_{f}(x)$.
  \item $U_{f}(x)g = U_{g^{-1}fg}(xg)$.
  \item Only finitely many of the orbitals~$U_{f}(y)$, as $y$~ranges
    through~$V$, are infinite.
  \item If $f$~and~$g$ commute and $U_{f}(x)$~is infinite, then
    $U_{f}(x)g = U_{f}(x)$.
  \end{enumerate}
\end{lemma}

\begin{proof}
  (i)~We have already observed that if $xf \neq x$, then the
  orbital~$U_{f}(x)$ is infinite.  Suppose $xf > x$, then $xf^{n+1} >
  xf^{n}$ for all $n \in \Z$.  So if $u \in U_{f}(x)$, there exist
  $m,n \in \Z$ such that $xf^{m} < u < xf^{n}$ where necessarily $m <
  n$.  Then $u < xf^{n} < uf^{n-m}$, which can only hold if $uf > u$.

  Part~(ii) is obtained by a similar argument to~(i), while
  part~(iii) is straightforward to establish from the definition.

  (iv)~Let $\set{U_{i}}{i \in I}$ be the set of those orbitals of~$f$
  that are infinite and suppose that $I$~is infinite.  Since
  the~$U_{i}$ are pairwise disjoint and each is a convex subset
  of~$V$, we can define, for each subset~$\Sigma$ of~$I$, an
  automorphism~$f_{\Sigma}$ of~$\Omega$ by
  \[
  vf_{\Sigma} = \begin{cases}
    vf &\text{if $v \in U_{i}$ where $i \in \Sigma$} \\
    v &\text{otherwise}.
  \end{cases}
  \]
  Since $f$~induces a non-identity transformation of each~$U_{i}$, we
  conclude that the~$f_{\Sigma}$ are distinct.  Hence, as $I$~has
  uncountably many subsets, we obtain a contradiction to the
  assumption that $\Aut\Omega$~is countable.  This establishes that
  only finitely many of the orbitals of~$f$ can be infinite.

  (v)~If $f$~and~$g$ commute, then part~(iii) of the lemma tells us
  that the action of~$g$ on~$V$ induces a permutation on the set of
  orbitals of~$f$.  Since only finitely many of these orbitals are
  infinite and since $g$~preserves the order on~$V$, it must be the
  case that $g$~fixes (setwise) all the orbitals of~$f$ that are
  infinite.
\end{proof}

\begin{lemma}
  \label{lem:orbitalinterval}
  Let $f \in \Aut\Omega$, \ $x \in V$ and suppose that the
  orbital~$U_{f}(x)$ is infinite.  If $a,b \in U_{f}(x)$ with $a < b$,
  then $\Aut (a,b) = \1$.
\end{lemma}

\begin{proof}
  Write~$B$ for the interval $(a,b) = \set{v \in V}{a < v < b}$.
  Since $a,b \in U_{f}(x)$, there exists some $m \in \Z$ such that $b
  < af^{m}$.  It follows that the sets~$Bf^{km}$, as $k$~ranges over
  the positive integers, are pairwise disjoint.  As $f$~is an
  automorphism of~$\Omega$, each set~$Bf^{km}$ is order-isomorphic
  to~$B$.  We now have an infinite number of pairwise disjoint convex
  subsets and so it follows that we can embed the Cartesian
  product~$\prod_{k=0}^{\infty} \Aut Bf^{km}$ in~$\Aut\Omega$ by
  extending automorphisms defined on each of the sets~$Bf^{km}$ to the
  whole set~$V$.  In view of the fact that $\Aut\Omega$~is countable,
  we deduce that $\Aut B = \1$.
\end{proof}

We are now able to establish one of our main steps along the way to
proving Theorem~\ref{thm:countableautoms}, namely that the infinite
orbitals~$U_{f}(x)$, as $f$~ranges over all automorphisms of~$\Omega$
and $x$~ranges over~$V$, are either disjoint or are equal.

\begin{prop}
  \label{prop:equalorb}
  Let $f$~and~$g$ be automorphisms of~$\Omega$, $x \in V$ and suppose
  that both orbitals $U_{f}(x)$~and~$U_{g}(x)$ are infinite.  Then
  $U_{f}(x) = U_{g}(x)$.
\end{prop}

\begin{proof}
  Suppose that $U_{f}(x) \neq U_{g}(x)$.  If $f$~has more than one
  infinite orbital, replace~$f$ by the map given by
  \[
  v \tilde{f} = \begin{cases}
    vf &\text{if $v \in U_{f}(x)$,} \\
    v &\text{otherwise}.
  \end{cases}
  \]
  Thus, we can assume that $f$~acts as the identity on~$V \setminus
  U_{f}(x)$ and that $U_{f}(v) = \{v\}$ for all $v \in V \setminus
  U_{f}(x)$.  Also, replacing~$f$ by~$f^{-1}$ if necessary, we can
  assume that $vf > v$ for all $v \in U_{f}(x)$.  Similarly, we can
  assume that $g$~has only one infinite orbital, namely~$U_{g}(x)$,
  and that $vg > v$ for all $v \in U_{g}(x)$.  We deal first with the
  possibility that one of these infinite orbitals is a proper subset
  of the other.  Without loss of generality, suppose $U_{f}(x) \subset
  U_{g}(x)$.  We shall consider the possible arrangements of the
  points in the complement~$U_{g}(x) \setminus U_{f}(x)$.

  First, if there exist $a,b \in U_{g}(x)$ such that $a < U_{f}(x) <
  b$, then note that $f$~induces a non-trivial automorphism of the
  interval~$(a,b)$.  (Indeed, $f$~acts non-trivially on the
  set~$U_{f}(x)$ and fixes all points in~$(a,b) \setminus U_{f}(x)$.)
  We then obtain a contradiction since Lemma~\ref{lem:orbitalinterval}
  applied to the orbital~$U_{g}(x)$ tells us that $\Aut (a,b)$~is
  trivial.  Hence no such pair $a$~and~$b$ exists.

  Therefore, if $U_{f}(x) \subset U_{g}(x)$, there exist points
  in~$U_{g}(x)$ greater than those in~$U_{f}(x)$ under the
  order~$\leq$, or points less than those in~$U_{f}(x)$, but not both.
  The argument for both cases is the same, so we shall assume the
  existence of some~$b \in U_{g}(x)$ with $U_{f}(x) < b$, but that
  there is no $a \in U_{g}(x)$ with $a < U_{f}(x)$.  In this setting,
  note first that if it were the case that $f$~and~$g$ commute, then
  $U_{f}(x)g = U_{f}(x)$ by Lemma~\ref{lem:orbital}(v), but this
  contradicts the fact that there exists some~$m$ such that $xg^{m} >
  b$.  Hence $f$~and~$g$ do not commute.

  Now for each~$v \in U_{g}(x)$, there is some $n \in \Z$ satisfying
  $vg^{n} > b$ and so $vg^{n} \notin U_{f}(x)$.  Equally, $vg^{m} < x$
  for some $m \in \Z$ and so $vg^{m} \in U_{f}(x)$ since
  $vg^{m}$~cannot satisfy $vg^{m} < U_{f}(x)$.  Then $vg^{m} <
  vg^{n}$, so that $m < n$.  It follows that for every $v \in
  U_{g}(x)$ there is a minimum integer~$m(v)$ satisfying $vg^{m(v)}
  \notin U_{f}(x)$ and this integer has the property that $vg^{n} \in
  U_{f}(x)$ for all $n < m(v)$ and $vg^{n} \notin U_{f}(x)$ for all $n
  \geq m(v)$.

  Now consider the automorphism~$\theta_{i}$ defined by $\theta_{i} =
  g^{i} f g^{-i}$, which by Lemma~\ref{lem:orbital}(iii) has a single
  infinite orbital, namely $U_{\theta_{i}}(xg^{-i}) = U_{f}(x)g^{-i}$,
  which is some subset of~$U_{g}(x)$ (since $U_{f}(x) \subseteq
  U_{g}(x)$ and $g$~fixes $U_{g}(x)$ setwise).  If $v \in U_{g}(x)$,
  observe $v \in U_{f}(x)g^{-i}$ if and only if $vg^{i} \in U_{f}(x)$;
  that is, when $i < m(v)$.  Consequently, $v\theta_{i} = v$ whenever
  $i \geq m(v)$ and, by Lemma~\ref{lem:orbital}(i), \ $v\theta_{i}
  \neq v$ whenever $i < m(v)$.

  Now if $\Sigma = \{ \sigma_{0}, \sigma_{1}, \sigma_{2}, \dots \}$ is
  an infinite subset of~$\N$ with $\sigma_{i} < \sigma_{i+1}$ for
  each~$i$, we can define another automorphism of~$\Omega$ by
  \[
  h_{\Sigma} = \lim_{n \to \infty} \theta_{\sigma_{n}} \dots
  \theta_{\sigma_{1}} \theta_{\sigma_{0}}.
  \]
  (In order to make sense of this definition, recall our convention is
  to write maps on the right.)  If $v \in V \setminus U_{g}(x)$, then
  $v\theta_{i} = v$ for all~$i$, so we observe $v\theta_{\sigma_{n}}
  \dots \theta_{\sigma_{0}} = v$ and hence $vh_{\Sigma}$~is defined and indeed
  equals~$v$ for such~$v$.  On the other hand, if $v \in U_{g}(x)$,
  then there exists some~$N$ such that $\sigma_{k} \geq m(v)$ for all
  $k > N$.  Thus $v\theta_{\sigma_{k}} = v$ for all such~$k$ and we
  conclude that
  \[
  v \theta_{\sigma_{n}} \dots \theta_{\sigma_{1}} \theta_{\sigma_{0}}
  = v \theta_{\sigma_{N}} \dots \theta_{\sigma_{1}} \theta_{\sigma_{0}}
  \]
  for all $n > N$.  Hence $vh_{\Sigma}$~is defined for all $v \in
  U_{g}(x)$ since $v\theta_{\sigma_{n}} \dots
  \theta_{\sigma_{0}}$~takes the same value independent of~$n$
  provided this~$n$ is large enough.  In addition to having observed
  that $h_{\Sigma}$~is well-defined, such calculations similarly show
  that $h_{\Sigma} \in \Aut\Omega$.

  Having verified that $h_{\Sigma}$~is defined for any
  (infinite)~$\Sigma \subseteq \N$, we now observe that $h_{\Sigma}
  \neq h_{T}$ for distinct $\Sigma,T \subseteq \N$.  Indeed, suppose
  $\Sigma = \{ \sigma_{0}, \sigma_{1}, \dots, \sigma_{r-1},
  \sigma_{r}, \dots \}$ and $T = \{ \sigma_{0}, \sigma_{1}, \dots,
  \sigma_{r-1}, \tau_{r}, \dots \}$ where, without loss of generality,
  $\sigma_{r} < \tau_{r}$.  Take $u = xg^{k}$ where $k = m(x) -
  \sigma_{r} - 1$.  Observe $ug^{\sigma_{r}} = xg^{m(x)-1} \in
  U_{f}(x)$ and $ug^{i} \notin U_{f}(x)$ for all $i > \sigma_{r}$.
  Thus $u \notin U_{f}(x)g^{-i}$ for all $i > \sigma_{r}$, so that
  $u\theta_{i} = u$ for all such~$i$.  Hence, for $n \geq r$,
  \[
  u \theta_{\sigma_{n}} \dots \theta_{\sigma_{0}} = u
  \theta_{\sigma_{r}} \dots \theta_{\sigma_{0}}
  \qquad \text{and} \qquad
  u \theta_{\tau_{n}} \dots \theta_{\tau_{r}} \theta_{\sigma_{r-1}}
  \dots \theta_{\sigma_{0}} = u \theta_{\sigma_{r-1}} \dots \theta_{\sigma_{0}}.
  \]
  As $u \in U_{f}(x)g^{-\sigma_{r}} = U_{\theta_{\sigma_{r}}}(x)$, we know
  $u\theta_{\sigma_{r}} \neq u$ and so we conclude $uh_{\Sigma} \neq
  uh_{T}$, which establishes our claim that the~$h_{\Sigma}$ are
  distinct.  Since $\Aut\Omega$~is countable, it cannot contain these
  uncountably many automorphisms~$h_{\Sigma}$ and we have another
  contradiction.  The other remaining case when $U_{f}(x) \subset
  U_{g}(x)$ is similar, which now establishes that $U_{f}(x)$~is not a
  proper subset of~$U_{g}(x)$ nor \emph{vice versa}.

  Thus there exists some $a \in U_{f}(x)$ and $b \in U_{g}(x)$ such
  that $a \notin U_{g}(x)$ and $b \notin U_{f}(x)$.  We may assume,
  without loss of generality that $a < U_{g}(x)$.  Then, since
  $U_{f}(x)$~and~$U_{g}(x)$ are convex, we observe $U_{f}(x) < b$.
  Moreover we also note that the sets $U_{f}(x) \setminus U_{g}(x)$,
  \ $U_{f}(x) \cap U_{g}(x)$ and~$U_{g}(x) \setminus U_{f}(x)$ are all
  convex and satisfy $U_{f}(x) \setminus U_{g}(x) < U_{f}(x) \cap
  U_{g}(x) < U_{g}(x) \setminus U_{f}(x)$.  Suppose first that
  $f$~and~$g$ commute.  Then as $b,x \in U_{g}(x)$, there exists some
  $n \in \Z$ such that $b < xg^{n}$.  However, this is impossible as
  $xg^{n} \in U_{f}(x)$ by use of Lemma~\ref{lem:orbital}(v).  Hence
  we it must be the case that $f$~and~$g$ do not commute.  Put $h =
  f^{-1}g^{-1}fg$, which is some non-identity element of~$\Aut\Omega$.
  If $v \in U_{f}(x) \setminus U_{g}(x)$, then $vf^{-1} < v <
  U_{g}(x)$ and so $vf^{-1} \notin U_{g}(x)$ and hence
  $vf^{-1}g^{-1}fg = vf^{-1}fg = vg = v$.  Similarly, if $vh = v$ for
  $v \in U_{g}(x) \setminus U_{f}(x)$.  It follows that any
  infinite~$U_{h}(y)$ is a subset of~$U_{f}(x) \cap U_{g}(x)$.
  However, we have already established that this is impossible, since
  a pair of non-identity automorphisms $f$~and~$h$ cannot have
  infinite orbitals satisfying $U_{h}(x) \subset U_{f}(x)$.  This
  final contradiction completes the proof of the claim: $U_{f}(x) =
  U_{g}(x)$.
\end{proof}

Recall that a \emph{linearly order group} is a group~$G$ together with
a linear order~$\leq$ upon it such that if $g,h,k \in G$ with $h \leq
k$, then $gh \leq gk$ and $hg \leq kg$.  An \emph{Archimedean group}
is a linearly ordered group~$G$ with the property if $g,h \in G$
satisfy $1 < g < h$, there exists $n \in \N$ such that $h < g^{n}$.

Let $f \in \Aut\Omega$ and fix $x \in V$ such that the orbital $U =
U_{f}(x)$~is infinite.  For $\phi,\psi \in \Aut U$, define $\phi \leq
\psi$ whenever $x\phi \leq x\psi$.  We shall observe that this is a
well-defined linear order with respect to which $\Aut U$~is an
Archimedean group.

\begin{lemma}
  \label{lem:Archimedean}
  \begin{enumerate}
  \item The map $\xi \colon \Aut U \to U$ given by $\phi \mapsto
    x\phi$ for each $\phi \in \Aut U$ is an injective map.
  \item The order~$\leq$ is a well-defined linear order on~$\Aut U$
    with respect to which $\Aut U$~is an Archimedean group.
  \end{enumerate}
\end{lemma}

\begin{proof}
  (i)~Suppose $\phi$~and~$\psi$ are distinct automorphisms of~$U$.
  Then $g = \phi\psi^{-1}$~can be extended to a non-identity
  automorphism of~$\Omega$ by defining $vg = v$ for all~$v \in V
  \setminus U$.  By assumption some $u \in U$ is moved by~$g$ and then
  $U_{g}(u) = U$ by Proposition~\ref{prop:equalorb}.  In particular,
  $x \in U_{g}(u)$ and hence $xg \neq x$ by
  Lemma~\ref{lem:orbital}(i)--(ii).  This shows that $x\phi \neq
  x\psi$, as is required to establish that $\xi$~is injective.

  (ii)~Part~(i) of the lemma shows that the set of automorphisms
  of~$U$ is in one-one correspondence with the subset
  $\set{x\phi}{\phi \in \Aut U}$ of~$U$ and hence the order on~$U$
  induces an order on~$\Aut U$; that is, the order~$\leq$ defined by
  $\phi \leq \psi$ if and only if $x\phi \leq x\psi$.  It is
  straightforward to verify that $\Aut U$~is a linearly ordered group
  with respect to~$\leq$.  (One makes use of Lemma~\ref{lem:orbital}
  in this verification.  For example, if $\phi,\psi,\theta \in \Aut U$
  with $\phi \leq \psi$, then $x\psi\phi^{-1} \geq x$ and use of
  Lemma~\ref{lem:orbital}(ii) shows that $x\theta\psi\phi^{-1} \geq
  x\theta$.  It then follows $\theta\phi \leq \theta\psi$, which is
  one of the facts that needs to be established.)

  Finally, if $1 < \phi < \psi$, we extend $\phi$~to an isomorphism
  of~$\Omega$ and observe $U = U_{\phi}(x)$ by
  Proposition~\ref{prop:equalorb}.  The definition of~$U_{\phi}(x)$
  then provides $n \in \N$ such that $x\psi < x\phi^{n}$, so $\psi <
  \phi^{n}$.  This establishes that $\Aut U$~is an Archimedean group
  with respect to the order~$\leq$.
\end{proof}

We can now make use of the result, originally due to
H\"{o}lder~\cite{Holder}, that an Archimedean group is isomorphic to
an additive subgroup of the set~$\R$ of real numbers (see, for
example,~\cite[Theorem~4.A]{Glass}).  In~\cite[Lemma~4.21]{Goodearl}
it is noted that such a subgroup is either cyclic or is a dense subset
of~$\R$.  Our current goal is to establish
Proposition~\ref{prop:AutU} below, namely that $\Aut U$~is an infinite
cyclic group, so let us assume, seeking a contradiction, that
$\oset{\Aut U}$~is a dense linearly ordered set.

In Lemma~\ref{lem:Archimedean} we have observed that the map~$\xi$ is
an order-isomorphism from~$\Aut U$ to the orbit of~$x$ under the
action of~$\Aut U$ (with the order on this orbit being that induced
from the ordered set~$\Omega$).  Thus $\set{x\phi}{\phi \in \Aut
  U}$~is a dense linearly ordered set with no maximum or minimum
element and is therefore order-isomorphic to~$\Qset$.  This
observation is independent of the choice of representative~$x$ in~$U$
and hence every orbit in~$U$ under the action of~$\Aut U$ is
order-isomorphic to~$\Qset$.

\begin{lemma}
  \label{lem:interleave}
  If $u,v \in U = U_{f}(x)$ and $\phi,\psi \in \Aut U$ with $\phi <
  \psi$, then there exists $\theta \in \Aut U$ with $u\phi < v\theta <
  u\psi$.
\end{lemma}

\begin{proof}
  By use of Lemma~\ref{lem:orbital}(ii), we observe that the
  hypothesis $\phi < \psi$ ensures that $u\phi < u\psi$.  When
  $u$~and~$v$ belong to the same orbit of~$\Aut U$ on~$U$ the claim is
  now immediate since that orbit is order-isomorphic to~$\Qset$.
  Suppose that $v$~is not in the orbit of~$u$ under the action
  of~$\Aut U$ and, by applying~$\phi^{-1}$ if necessary, assume that
  $\phi$~is the identity automorphism.  Thus $\psi$~is a non-identity
  automorphism of~$U$ satisfying $u < u\psi$ and we must find~$\theta
  \in \Aut U$ with $u < v\theta < u\psi$.

  Suppose first that $v < u$.  We extend~$\psi$ to an automorphism
  of~$\Omega$ by defining it to fix all points outside the
  orbital~$U$.  Then Proposition~\ref{prop:equalorb} tells us that
  $U_{\psi}(u) = U_{f}(u) = U$.  In particular, there exists some~$n
  \in \N$ such that $v\psi^{n} > u$.  Take~$n$ to be the minimum
  positive integer satisfying $v\psi^{n} > u$.  Then $v\psi^{n-1} <
  u$, so $u < v\psi^{n} < u\psi$ and so, in this case, $\theta =
  \psi^{n}$ is our required automorphism.

  If $u < v$, then since $U_{\psi}(u) = U$ we can find some power
  of~$\psi$ such that $v\psi^{m} < u$.  Applying the previous
  paragraph to~$v\psi^{m}$ finds $n \in \N$ such that $u < v\psi^{m+n}
  < u\psi$ and then $\theta = \psi^{m+n}$ is the automorphism we seek.
\end{proof}

Now enumerate the points in~$U$ as the sequence~$(x_{n})$.  First
consider the set~$\T_{0}$ of convex subsets~$S$ of~$U$ such that
(i)~$S$~contains~$x_{0}$ and (ii)~$S\phi$~is disjoint from~$S$ for
every non-identity automorphism~$\phi$ of~$U$.  As only the identity
automorphism fixes~$x_{0}$ (see Lemma~\ref{lem:orbital}(i)--(ii)) we
conclude $\{x_{0}\}$~is a set in~$\T_{0}$ (so $\T_{0}$~is non-empty)
and it is straightforward to verify that the union of any chain of
subsets of~$\T_{0}$ is again a member of~$\T_{0}$.  Hence, by Zorn's
Lemma, there is some maximal member~$M_{0}$ of~$\T_{0}$.

Suppose then that, for some~$n$, we have found subsets
$M_{0}$,~$M_{1}$, \dots,~$M_{k}$ of~$U$ such that
\begin{itemize}
\item $x_{0},x_{1},\dots,x_{n-1} \in \set{u\phi}{u \in M_{0} \cup
  \dots \cup M_{k}, \; \phi \in \Aut U}$, and
\item $M_{i}$~is a maximal convex subset of~$U$ subject to
  \begin{align}
    &M_{i} \subseteq U \setminus \set{u\phi}{u \in M_{0} \cup \dots
      \cup M_{i-1}, \; \phi \in \Aut U}
    \label{eq:Mi-domain} \\
    &M_{i}\phi \cap M_{i} = \emptyset \quad \text{for every
      non-identity automorphism~$\phi$ of~$U$}.
    \label{eq:Mi-autom}
  \end{align}
\end{itemize}
(Note that Condition~\eqref{eq:Mi-domain} ensures that $M_{i}$~is
disjoint from every translate~$M_{j}\phi$ of a previously defined
subset, with $1 \leq j < i$, under some automorphism of~$U$.)

If $x_{n}$~is already the image of some point in $M_{0} \cup \dots
\cup M_{k}$ under some automorphism of~$U$, then we need create no new
subset~$M_{i}$ at this stage.  Otherwise, consider the set~$\T_{k+1}$
of subsets~$S$ of $U \setminus \set{u\phi}{u \in M_{0} \cup \dots \cup
  M_{k}, \; \phi \in \Aut U}$ such that (i)~$S$~is a convex subset
of~$U$, \ (ii)~$x_{n} \in S$, and (iii)~$S\phi$~is disjoint from~$S$
for every non-identity automorphism~$\phi$ of~$U$.  Again, an
application of Zorn's Lemma provides the existence of a maximal
member~$M_{k}$ in~$\T_{k+1}$.

In this way, we find a family~$(M_{i})$ of convex subsets of~$U$,
indexed by some set~$I$ (where either $I = \N$ or $I = \{ 0,1,\dots,k
\}$ for some~$k$), such that $U$~is the disjoint union of the
sets~$M_{i}\phi$, for $i \in I$ and $\phi \in \Aut U$, and $M_{i}$~is
maximal among convex subsets of~$U$ satisfying
\eqref{eq:Mi-domain}~and~\eqref{eq:Mi-autom} above.  As convex subsets
of the linearly ordered set~$U$, there is an induced order on the
sets~$\set{M_{i}\phi}{i \in I, \; \phi \in \Aut U}$.

\begin{lemma}
  \label{lem:Mi-order}
  Suppose that $M_{i_{1}}\phi_{1} < M_{i_{2}}\phi_{2}$ for some
  $i_{1},i_{2} \in I$ and some $\phi_{1},\phi_{2} \in \Aut U$.  Then
  for each $j \in I$, there exists some $\psi \in \Aut U$ with
  \[
  M_{i_{1}}\phi_{1} < M_{j}\psi < M_{i_{2}}\phi_{2}.
  \]
\end{lemma}

\begin{proof}
  Suppose first that $i_{1} = i_{2}$.  Pick $u \in M_{i_{1}}$ and $v
  \in M_{j}$.  By Lemma~\ref{lem:interleave}, there exists $\psi \in
  \Aut U$ such that $u\phi_{1} < v\psi < u\phi_{2}$.  Hence, as the
  sets concerned are convex, $M_{i_{1}}\phi_{1} < M_{j}\psi <
  M_{i_{1}}\phi_{2} = M_{i_{2}}\phi_{2}$, as required.

  It remains to deal with the case when $i_{1} \neq i_{2}$.  If there
  exists some~$k \in I$ and automorphisms $\theta_{1},\theta_{2} \in
  \Aut U$ with $M_{i_{1}}\phi_{1} \leq M_{k}\theta_{1} <
  M_{k}\theta_{2} \leq M_{i_{2}}\phi_{2}$, then the previous paragraph
  can be applied to $M_{k}\theta_{1} < M_{k}\theta_{2}$ and we would
  have established the required result.  Seeking a contradiction, let
  us assume that no such~$k$, $\theta_{1}$ and~$\theta_{2}$ exist.  As
  a consequence, we conclude that there is no $\theta \in \Aut U$ with
  $M_{i_{1}}\phi_{1} < M_{i_{1}}\theta < M_{i_{2}}\phi_{2}$ or with
  $M_{i_{1}}\phi_{1} < M_{i_{2}}\theta < M_{i_{2}}\phi_{2}$ and that,
  for each $k \in I$, there is at most one $\theta \in \Aut U$ with
  $M_{i_{1}}\phi_{1} < M_{k}\theta < M_{i_{2}}\phi_{2}$.

  Write $K$~for the set of those~$k \in I$ for which there exists
  $\theta_{k} \in \Aut U$ with $M_{i_{1}}\phi_{1} \leq M_{k}\theta_{k}
  \leq M_{i_{2}}\phi_{2}$.  (So, in particular, $i_{1},i_{2} \in K$
  and that $\theta_{i_{m}} = \phi_{m}$ for $m = 1$,~$2$.)  Let $m$~be
  the smallest integer in~$K$.  By applying the inverse
  of~$\theta_{m}$ if necessary, there is no loss of generality in
  assuming that $\theta_{m}$~is the identity map.  Put $S = \bigcup_{k
    \in K} M_{k}\theta_{k}$, so that $M_{m}$~is a proper subset of~$S$
  by our assumption on~$\theta_{m}$.  Since $U$~is the union of all
  translates~$M_{j}\theta$, it follows that every point between
  $M_{i_{1}}\phi_{1}$~and~$M_{i_{2}}\phi_{2}$ lies in
  some~$M_{k}\theta_{k}$ with $k \in K$ and we deduce that the set~$S$
  is convex.  The set~$S$ is also disjoint from all translates
  of~$M_{j}$ for $j < m$, since each set~$M_{k}$ for $k \in K$
  satisfies~\eqref{eq:Mi-domain} above, while $S\psi \cap S =
  \emptyset$ for every non-identity~$\psi \in \Aut U$ since each
  set~$M_{k}$ satisfies~\eqref{eq:Mi-autom}.  We now have a
  contradiction to~$M_{m}$ being a maximal convex subset satisfying
  \eqref{eq:Mi-domain}~and~\eqref{eq:Mi-autom}.  This contradiction
  completes the proof of the lemma.
\end{proof}

The property given in Lemma~\ref{lem:Mi-order} will essentially
characterise the structure of the ordered set~$\oset{U}$.  To describe
this fully, we first introduce a new relational structure.

Let $I$~be a countable set.  We define an \emph{$I$\nbd coloured
  linearly ordered set} to be a relational structure $\Gamma = (V,
\leq, (R_{i})_{i \in I})$ where $\leq$~is a linear order on the
set~$V$ and where each~$R_{i}$ is a binary relation on~$V$ of the
form~$R_{i} = V_{i} \times V_{i}$ such that $V$~is the disjoint union
of the sets~$V_{i}$.  Thus the sequence~$(R_{i})_{i \in I}$ encodes an
equivalence relation on~$V$ with equivalence classes~$V_{i}$, for $i
\in I$, in such a way that any automorphism of~$\Gamma$ fixes each of
the equivalence classes setwise.

The class of finite $I$\nbd coloured linearly ordered sets satisfies
the hereditary property, the joint embedding property and the
amalgamation property and therefore this class possesses a unique
Fra\"{\i}ss\'{e} limit $\Q_{I} = (W, {\leq}, (R_{i})_{i \in I} )$.
Write $W_{i}$~for the equivalence class determined by the
relation~$R_{i}$.  This structure is characterised by the following
property: $(W, {\leq})$~is a countable linearly ordered set without
maximum or minimum elements such that for every pair $u,v \in W$ with
$u < v$ and every $i \in I$ there exists $w \in W_{i}$ with $u < w <
v$.  Indeed, it can be shown by a back-and-forth argument that any two
countable structures satisfying this condition are isomorphic as
$I$\nbd coloured linearly ordered sets (and again such an isomorphism
takes the equivalence class in the first structure indexed by~$i \in
I$ to that in the second indexed by~$i$).  We shall call this
Fra\"{\i}ss\'{e} limit the \emph{$I$\nbd coloured ordered set of
  rational numbers} in view of the fact that $(W,\leq)$~is
order-isomorphic to~$\Qset$.  In view of this order-isomorphism, we
shall rename the set~$W$ as~$\Q$, so that the $I$\nbd coloured
linearly ordered set is denoted $\Q_{I} = (\Q, {\leq}, (R_{i})_{i \in
  I} )$.

\begin{prop}
  The automorphism group of the $I$\nbd coloured ordered set~$\Q_{I}$
  of rational numbers is uncountable.
\end{prop}

\begin{proof}
  Note that $\Aut \Q_{I}$~is non-trivial since given any $i \in I$ and
  two points $x,y \in W_{i}$, a back-and-forth argument establishes
  the existence of an order-isomorphism that preserves the equivalence
  classes~$W_{i}$ and maps $x$ to~$y$.  The following argument extends
  this to show in fact there are uncountably many automorphisms
  of~$\Q_{I}$.
  
  We shall write $\Z \times \Q_{I}$~for the $I$\nbd coloured linearly
  ordered set defined as follows: as an ordered set it is the set~$\Z
  \times \Q$ equipped with the lexicographic order; that is, $(m,x)
  \leq (n,y)$ if and only if $m < n$, or $m = n$ and $x \leq y$.  To
  colour~$\Z \times \Q_{I}$, for each $i \in I$, the $i$th equivalence
  class is~$\Z \times W_{i}$ where $W_{i}$~is the $i$th equivalence
  class in~$\Q_{I}$.  In effect, with $\Z \times \Q_{I}$, we are
  taking countably many copies of~$\Q_{I}$, placing them in sequence
  in terms of the order, and then taking the $i$th equivalence classes
  in each copy of~$\Q_{I}$ together to form a single equivalence class
  in~$\Z \times \Q_{I}$.

  One observes that $\Z \times \Q_{I}$~is a countable linearly ordered
  set with no maximum or minimum element and that it has the property
  that for each $u,v \in \Z$ with $u < v$ and all $i \in I$, there
  exists $w \in \Z \times W_{i}$ with $u < w < v$.  Thus, $\Z \times
  \Q_{I}$~satisfies the defining property of~$\Q_{I}$ so that $\Z
  \times \Q_{I} \cong \Q_{I}$ as $I$\nbd coloured linearly ordered
  sets.

  If $\mathbf{f} = (f_{n})$ is a sequence of automorphisms of the
  structure~$\Q_{I}$, we can define $\hat{\mathbf{f}} \in \Aut ( \Z
  \times \Q_{I} )$ by $(n,x)\hat{\mathbf{f}} = (n,xf_{n})$ for each $n
  \in \Z$ and $x \in \Q$.  This defines an injective map $\mathbf{f}
  \mapsto \hat{\mathbf{f}}$ from the Cartesian product
  $\prod_{n=1}^{\infty} \Aut \Q_{I}$ to $\Aut(\Z \times \Q_{I})$.  It
  now follows that $\Aut \Q_{I} \cong \Aut ( \Z \times \Q_{I} )$ is
  indeed uncountable.
\end{proof}

We now return to the automorphism group of the orbital $U = U_{f}(x)$
under our current assumption that $\Aut U$~is order-isomorphic to some
dense linearly ordered set.  Recall that we have defined a sequence
$\mathbf{M} = (M_{i})_{i \in I}$ of convex subsets of~$U$ indexed
by~$I$.  We shall use the $I$\nbd coloured ordered set~$\Q_{I}$ of
rational numbers, where $I$~is the set indexing our convex subsets.
Recall that the equivalence classes on~$\Q$ associated to this
relational structure are denoted~$(W_{i})_{i \in I}$.  Now write
$\Q_{I}(\mathbf{M})$ for the ordered set~$\oset{S}$ where $S =
\bigcup_{i \in I} (W_{i} \times M_{i})$ and the order~$\leq$ is the
lexicographic order (that is, $(x,m) \leq (y,n)$ if and only if $x <
y$, or $x = y$ and $m \leq n$).  Now if $\phi \in \Aut \Q_{I}$ (a
colour- and order-preserving bijection of this structure), we can
define an automorphism~$\tilde{\phi}$ of~$\Q_{I}(\mathbf{M})$ by
$(x,m)\tilde{\phi} = (x\phi,m)$ for $(x,m) \in S$.  Note that we rely
upon the fact that $\phi$~preserves the $i$th equivalence
class~$W_{i}$ when observing that $\tilde{\phi}$~is indeed a
well-defined map.  The following now follows from the fact that $\Aut
\Q_{I}$~is uncountable.

\begin{cor}
  \label{cor:AutQM}
  The automorphism group of the ordered set $\Q_{I}(\mathbf{M}) =
  \oset{S}$ is uncountable. \qed
\end{cor}

Let us now consider the set $\mathcal{M} = \set{M_{i}\phi}{i \in I, \;
  \phi \in \Aut U}$ of all translates of the sets~$M_{i}$ under the
action of the automorphism group of~$U$ and view this as an ordered
set using the order induced on these convex subsets from the order
on~$U$.  We shall also define a relation~$R_{i}$ to be that relating
the points~$M_{i}\phi$ for $\phi \in \Aut U$, so that $W_{i} =
\set{M_{i}\phi}{\phi \in \Aut U}$ is the corresponding subset
of~$\mathcal{M}$ indexed by~$i$.  Since $\Aut U$~has no maximum or
minimum element, the same is true of~$\mathcal{M}$ and now
Lemma~\ref{lem:Mi-order} tells us that $(\mathcal{M},{\leq},(R_{i})_{i
  \in I})$ satisfies the defining property of the $I$\nbd coloured
ordered set~$\Q_{I}$ of rational numbers.  Thus these structures are
isomorphic as $I$\nbd coloured ordered sets.  Returning to our
set~$U$, we now observe that we can reconstruct this set
from~$\mathcal{M}$ by replacing each point~$M_{i}\phi$ by a copy of
the ordered set~$M_{i}$.  This tells us that $\oset{U}$~is
order-isomorphic to~$\Q_{I}(\mathbf{M})$.  This now gives us the
contradiction that we seek: Corollary~\ref{cor:AutQM} tells us that
$\Aut U$~is uncountable, which is contrary to our running assumption.

In conclusion, we have now established our final step towards the main
result of this section.

\begin{prop}
  \label{prop:AutU}
  Let $f \in \Aut \Omega$, \ $x \in V$ and suppose $U = U_{f}(x)$ is
  infinite.  Then $\Aut U$~is an infinite cyclic group. \qed
\end{prop}

Putting all our work together, we can now establish
Theorem~\ref{thm:countableautoms}.

\begin{proofof}{Theorem~\ref{thm:countableautoms}}
  Consider the set~$\set{U_{i}}{i \in I}$ of all subsets of~$V$ that
  arise as an infinite orbital of some automorphism of~$\Omega$.
  Proposition~\ref{prop:equalorb} tells us that these sets~$U_{i}$ are
  pairwise disjoint.  Moreover, if $f_{i} \in \Aut U_{i}$ for each $i
  \in I$, then there is an extension~$f$ to an automorphism
  of~$\Omega$ by
  \[
  vf = \begin{cases}
    vf_{i} &\text{if $v \in U_{i}$ for some $i \in I$,} \\
    v &\text{otherwise}.
  \end{cases}
  \]
  Since any automorphism of~$\Omega$ must fix all points in $V
  \setminus \left( \bigcup_{i \in I} U_{i} \right)$, we conclude that
  $\Aut\Omega$~is isomorphic to the Cartesian product of the
  automorphism groups of the~$U_{i}$.  The countability
  of~$\Aut\Omega$ combined with Proposition~\ref{prop:AutU} tells us
  that $I$~is finite and that $\Aut\Omega \cong \Z^{\card{I}}$.  This
  completes the proof of our theorem.
\end{proofof}

\paragraph{Funding:} This work was supported by the Engineering and
Physical Sciences Research Council (EPSRC) [Doctoral Training Grant to
  J.D.McPhee, EP/H011978/1 to M.Quick.].

\paragraph{Acknowledgements:} The authors thank Igor Dolinka for
drawing their attention to the results of Kubi\'{s} in~\cite{Kubis}.
We also thank the anonymous referee for their careful reading of the
paper and suggestions.

\end{document}